\documentclass[12pt,oneside]{amsart}
\usepackage{t1enc}
\usepackage[latin2]{inputenc}

\usepackage{amsmath,amsfonts}
\usepackage{amsthm}
\usepackage{amscd}
\usepackage{amssymb}
\usepackage{enumerate}
\usepackage{mathrsfs}
\usepackage{dsfont}

\usepackage{stmaryrd}
\usepackage[T1]{fontenc}

\usepackage{mathtools}

\usepackage{bbm}

\usepackage{tikz}
\usepackage{tikz-cd}
\usepackage{wasysym}
\usepackage{verbatim}
\usepackage{fancyhdr}
\usepackage{fancybox}
\usepackage{url}
\usepackage{xcolor}

\usepackage[pdfencoding=auto, psdextra]{hyperref}
\hypersetup{colorlinks=true, citecolor=black, linkcolor=black, 
bookmarksopen=true,
pdftitle={On the Rudin-Blass ordering of measures},
}

\newcommand{\dom}{\mathrm{dom}}

\newcommand{\bc}{\begin{center}}
\newcommand{\ec}{\end{center}}

\newcommand{\restr}{\mathord{\upharpoonright}}

\DeclareMathOperator{\supp}{supp}

\newcommand{\F}{\mathcal{F}}
\newcommand{\RR}{\mathbb{R}}

\newcommand{\forces}[2]{\Vdash_{#2}``#1"}

\newtheorem{thm}{Theorem}[section]
\newtheorem*{thm*}{Theorem}
\newtheorem{lem}[thm]{Lemma}
\newtheorem{prop}[thm]{Proposition}
\newtheorem{cor}[thm]{Corollary}
\newtheorem{fact}[thm]{Fact}

\theoremstyle{definition}

\newtheorem{prob}[thm]{Problem}

\newtheorem{df}[thm]{Definition}
\newtheorem{exa}[thm]{Example}

\newtheorem{rem}[thm]{Remark}

\newtheorem*{mainthm1}{Theorem \ref{thm:RBMinNotQ}}
\newtheorem*{mainthm2}{Theorem \ref{thm:RBMinDichotomy}}

\title[On the Rudin-Blass ordering of measures]{On the Rudin-Blass ordering of measures}

\author{Piotr Borodulin--Nadzieja}
\address[Piotr Borodulin-Nadzieja]{Mathematical Institute, University of Wroc\l aw \\   pl. Grunwaldzki 2, 50-384 Wroc\l aw, Poland}
\email{pborod@math.uni.wroc.pl}

\author{Arturo Mart\'{i}nez-Celis}
\address[Arturo Martinez-Celis]{Mathematical Institute, University of Wroc\l aw \\   pl. Grunwaldzki 2, 50-384 Wroc\l aw, Poland}
\email{arturo.martinez-celis@math.uni.wroc.pl}

\author{Adam Morawski}
\address[Adam Morawski]{Faculty of Mathematics and Physics, Charles University \\ Ke Karlovu 3, 121 16 Praha 2, Czech Republic}
\address{Institute of Mathematics, Czech Academy of Sciences \\ \ \v{Z}itn\'{a} 25, 110 00 Praha 1, Czech Republic}
\email{morawski@math.cas.cz}

\author{Jadwiga \'Swierczy\'nska}
\address[Jadwiga \'Swierczy\'nska]{Institute of Computer Science, University of Wroc\l aw \\   pl. Grunwaldzki 2, 50-384 Wroc\l aw, Poland}
\email{330498@uwr.edu.pl}

\thanks{The third author was was supported by the 
The University Research Centers of Charles University (UNCE/24/SCI/022),
The Czech Academy of Sciences CAS (RVO 67985840),
Charles University Grant Agency (GAUK project no. 219025)}


\begin{document}

\begin{abstract} 
	We study the Rudin-Blass (and the Rudin-Keisler) ordering on the finite additive measures on $\omega$. We propose a generalization of the notion of Q-point and selective ultrafilter to measures: Q-measures and selective measures. We show some symmetries	between Q-points and Q-measures but also we show where those symmetries break up. In particular we present an example of a measure which is minimal in the sense of Rudin-Blass but which is not a Q-measure.
\end{abstract}

\maketitle

\section*{Introduction}

By a measure $\mu$ on $\omega$ we mean a finitely additive set function of values in $[0,1]$ and such that $\mu(\omega)=1$. Every ultrafilter $\mathcal{U}$ on the natural numbers can be seen as an example of a measure, by assigning $1$ to elements of $\mathcal{U}$, and $0$ to the elements of the dual ideal. Some measures, the non-atomic ones, are far from being ultrafilters. Perhaps the most natural example of such measure is the following:
\[ \mu(A) = \lim_{n\to \mathcal{U}} \frac{|A \cap n|}{n}, \]
where $\mathcal{U}$ is an ultrafilter on $\omega$. 

Presumably, at the beginning of the history of ultrafilters, the non-atomic measures on $\omega$, particularly the invariant ones, were studied as intensively as the atomic ones (i.e. the ultrafilters), see e.g. \cite{Chou}, \cite{Fairchild}, \cite{Choquet}. 
At some point, the theory of ultrafilters flourished and expanded in many different directions, whereas the non-atomic measures were studied with a relatively modest intensity. These were investigated by Fremlin and Talagrand in \cite{FremlinTalagrand}, van Douwen in \cite{Douwen}, and other people (see more references below). 

The structure of ultrafilters on $\omega$ can be investigated in different contexts. One of the most fruitful approaches is through Rudin-Keisler and Rudin-Blass orderings (see \cite{Laflamme} and \cite{Blass73}). Loosely speaking, an ultrafilter $\mathcal{U}$ is smaller (in
Rudin-Blass or Rudin-Keisler order) than $\mathcal{V}$ if $\mathcal{U}$ contains no more information than $\mathcal{V}$. 

The Rudin-Blass and Rudin-Keisler orderings can be easily generalized for measures, and it seems that this generalization was only considered for the first time very recently, in \cite{cancino}, where it was used to study measures in the Silver extension. The authors
proved that, consistently, there is a measure that does not have an ultrafilter Rudin-Keisler below it. This suggests that the order structure of measures does not necessarily come down to the ultrafilter case.

The most natural question concerning the RB and RK orderings is the one concerning their minimal elements.  
For ultrafilters, the Rudin-Blass minimal ultrafilters are exactly the Q-points and the Rudin-Keisler ones are exactly the selective ultrafilters.
An ultrafilter $\mathcal{U}$ is a \emph{Q-point} if every partition of $\omega$ into finite sets has a selector in $\mathcal{U}$. An ultrafilter $\mathcal{U}$ is \emph{selective} if every partition into sets from the dual ideal has a selector in $\mathcal{U}$.
These notions have natural
generalizations to the realm of measures that, to the best of our
knowledge, have not been studied yet (with a very recent exception of \cite{Luis} which we will mention later). 

We say that a measure $\mu$ on $\omega$ is a \emph{Q-measure} if every partition of $\omega$ into finite sets has a selector of full measure. A measure $\mu$ is \emph{selective} if for every partition $\langle A_n\rangle_{n\in\omega}$ of $\omega$ there is a selector $S$ of the biggest
possible measure (i.e. $\mu(S) = 1 - \sum_n \mu(A_n)$).

We managed to prove that there are strong symmetries between the theory of Q-measures, selective measures and P-measures (generalizations of P-points, see below). In particular, we proved the following results which are generalizations of the
fact well known for ultrafilters:

\begin{itemize}
	\item under CH (in fact, under $\mathrm{cov}(\mathcal{M}) = \mathfrak{c}$) there are non-atomic selective measures (Theorem \ref{covM=c}), 
	\item a measure is selective if and only if it is simultaneously a P-measure and a Q-measure (Proposition \ref{lem:selective_p_q}),
	\item every Q-measure (selective measure) is minimal in Rudin-Blass (Rudin-Keisler, respectively) ordering (Theorem \ref{thm:q_measures_minimal_rb}),
	\item the generic forcing for adding a non-atomic measure adds a selective measure (Proposition \ref{prop:generic_selective}).
\end{itemize}

There are also substantial differences. The main difference we found is that the property of being Q-measure is not equivalent to minimality in the Rudin-Blass ordering. 

\begin{mainthm1} Assuming CH (or $\mathfrak{p}=\mathfrak{c}$), there is a Rudin-Blass minimal measure which is not a Q-measure (and there is an Rudin-Keisler minimal measure which is not selective).
\end{mainthm1} 

This measure has interesting properties. It is far from being a Q-measure: it is shift invariant and Lebesgue measurable (as a function from the Cantor set to the reals). Q-measures are Rudin-Blass minimal because its filters of measure 1 sets are
sufficiently close to Q-points. The reason for the minimality of the above measure is more delicate and follows from its 'fractal' structure. 

We also found an interesting dichotomy concerning a weakening of a Q-measure: A measure is a \emph{Q$^+$-measure} if for each partition into finite sets there is a selector of positive measure (this is a natural counterpart of Q$^+$ filters). We found the following result:

\begin{mainthm2} A Rudin-Blass minimal measure is either a Q$^+$-measure or it is Rudin-Blass equivalent to a shift invariant measure.
\end{mainthm2}

So, a measure which is Rudin-Blass minimal but is not Q$^+$ has to be (equivalent to a measure) symmetric with respect to the natural semigroup operation. Moreover, if such a measure is additionally a P-measure (see below), then it can be decomposed
into a Q$^+$-part and a shift invariant part (Corollary \ref{decomposition}).

The Q-points and selective ultrafilters are often put in one line with another type of ultrafilter: a P-point (e.g. selective ultrafilters are exactly those which are Q-points and P-points simultaneously). This notion has been generalized to the measure context in 1984, by Mekler (see \cite{mekler}). 
We call this generalization a P-measure (originally it was called measures with AP property). It turned out to be an interesting notion which attracted attention of several mathematicians (see \cite{blassple}, \cite{grebik}, \cite{Kunisada}, \cite{PbnDamian}).

It is easy to see that the existence of a P-point implies the existence of a non-atomic P-measure. It was a longstanding open problem (until recently) if the existence of P-measures implies the existence of P-points. It turns out that the answer is negative. Interestingly, the situation for Q-measures is quite the opposite. We have summarized this duality in the following table.
\bigskip

\begin{center}
	\begin{tabular}{ | c | c | c | }
		\hline
		Type & $\exists$ ult. $\implies$ $\exists$ non-atomic measure &  $\exists$ measure $\implies$ $\exists$ ult. \\
		\hline
		P & yes (\cite{mekler}) & no (\cite{cancino}) \\
		\hline
		Q & no (Corollary \ref{heike}) & yes (Proposition \ref{q-m-q-p}) \\
		\hline
		selective & no (Corollary \ref{heike}) & no (Theorem \ref{s-m-n-u})  \\
		\hline
	\end{tabular}
\end{center}
\bigskip
	
While we were working on this article we were informed by Antonio Avil\'es about the research in \cite{Luis}, in which the authors also defined Q-measures (although they were interested more in Q$^+$-measures). Interestingly, the authors of 
\cite{Luis} were motivated by a natural problem in Banach space theory. We overview some of their results in Section \ref{sec:Q}. In Section \ref{sec:Q+} we show that the existence of measures satisfying several conditions weaker than the Q property
considered  in \cite{Luis} is equivalent, answering \cite[Question 3]{Luis}.

In Section \ref{sec:preliminaries} we introduce basic definitions and notation. In Section \ref{Schemes} we present some ways of constructing measures on $\omega$ which will be used later on in the article. In Section \ref{sec:pqs} we define
Rudin-Blass and Rudin-Keisler orderings on measures. In Section \ref{sec:qmeasure} we introduce the special types of measures: P-measures, Q-measures, and selective measures, we discuss their basic properties and their existence. Section \ref{sec:minimality} is devoted mostly to
the construction of a RB-minimal measure which is shift invariant.

\section{Preliminaries}\label{sec:preliminaries}

By a \emph{measure} on $\omega$ we always mean a finitely additive probability measure. We will also assume that it vanishes on points, i.e. $\mu(F)=0$ for every finite $F\subseteq \omega$. We will call such measures \emph{free} but we will often 
assume this property without mentioning it. If $\mathcal{U}$ is
an ultrafilter, then the function $\delta_\mathcal{U}$ assigning $1$ to elements of $\mathcal{U}$ and $0$ to the rest, is a measure. Sometimes we will just treat ultrafilters as measures. Notice that an ultrafilter is free if and only if it is free as a measure.

If $A, B \subseteq \omega$, then $A \subseteq^* B$ denotes that $|A \setminus B| < \omega$. Usually we will not distinguish sets which are equal modulo finite sets (and so, formally, we will work in $\mathcal{P}(\omega)/fin$). A \emph{pseudo-intersection} of a family $\mathcal{A} \subseteq \mathcal{P}(\omega)$ is a set $X\subseteq \omega$ such that $X\subseteq^* A$ for each $A\in \mathcal{A}$.

Every finitely additive measure on a Boolean algebra $\mathbb{A}$ can be uniquely extended to a $\sigma$-additive measure on the Stone space of $\mathbb{A}$, see \cite[Corollary 2.8]{Plebanek2024}. Also, every $\sigma$-additive measure defined on a compact zero-dimensional space $X$
induces uniquely a finitely additive measure on the Boolean algebra of clopen subsets of $X$. In particular, free measures on $\omega$ can be extended to $\sigma$-additive measures on $\beta\omega\setminus \omega$ and every
measure on $\beta\omega\setminus \omega$ induces a free measure on $\omega$.

A measure $\mu$ on a Boolean algebra $\mathbb{A}$ is \emph{non-atomic} if for every $\varepsilon>0$ there is a partition $\langle A_n\rangle_{n<N}$ of the unity such that $\mu(A_n)<\varepsilon$ for every $n<N$. Notice that a measure $\mu$ is non-atomic if and only if
its extension to $\sigma$-additive measure on the Stone space vanishes on points.

A measure $\mu$ on $\omega$ is \emph{shift invariant} if $\mu(A+1) = \mu(A)$ for each $A\subseteq \omega$, where $A + 1 = \{n+1\colon n\in A\}$.

By $\mathbb{B}$ we will denote the measure algebra $\mathbb{B} = \mathrm{Bor}(2^\omega)/_{\lambda = 0}$, where $\lambda$ is the Lebesgue measure on $2^\omega$. The measure $\lambda$ can be naturally transported to $\mathbb{B}$, and the resulting measure is also denoted by $\lambda$.

One can consider bounded signed measures on $\omega$ (i.e. finitely additive measures $\mu\colon \mathcal{P}(\omega) \to \mathbb{R}$). It is possible to integrate elements of $\ell^\infty$ with respect to such measures, and there is even a general
theory of integration with respect to finitely additive vector measures on Boolean algebras (see e.g. \cite[Chapter 29]{Schechter}). In the case of measures on $\omega$, one can define the integration rather straightforward by following the classical
Lebesgue construction. Notice that the integral in this context is a functional extending the limit, i.e. if $\lim_{n\rightarrow \infty} f(n) = r$, then $\int f \ d\mu = r$. 

Let $\mathcal{M}$ be the set of all bounded signed measures on $\omega$.
This set can be endowed with the norm $\|\mu\|=\sup\{\mu(A)-\mu(\omega\setminus A):\, A\subseteq \omega\}$, forming a real Banach space. It can be shown that $\mathcal{M} \cong (\ell_\infty)^*$. More precisely, the following transformations are
isomorphism (given $\mu\in \mathcal{M}$ and $\varphi\in (\ell_\infty)^*$):
\[\varphi_\mu(x)=\int x\,\mathrm{d}\mu, \quad\quad \mu_\varphi(A)=\varphi(\chi_A).\]
More details on this construction can be found in \cite{TheoryOfCharges}.

By $\mathcal{M}_1$ we will denote the unit ball of $\mathcal{M}$. Notice that all the probability measures, so the measures we are interested in, belong to $\mathcal{M}_1$.

The other natural topology on $\mathcal{M}$ is the weak$^*$ topology. Recall that the sets of the form
\[ K(\varepsilon, \mu, f) = \left\{ \nu \colon \left|\int f d \mu - \int f d \nu\right|<\varepsilon\right\} \]
for $\varepsilon>0$, $\mu \in \mathcal{M}$ and a bounded function $f\colon \omega \to \omega$, form a subbasis for this topology. In fact, we may assume that $f$ is a characteristic function so the sets
\[ K(\varepsilon, \mu, A) = \{ \nu \colon |\mu(A) - \nu(A)|<\varepsilon\} \]
for $\varepsilon>0$, $\mu \in \mathcal{M}$ and $A\subseteq \omega$, also form a subbasis for the same topology.

To finish this section, we present the following basic property of measures.

\begin{prop}\label{prop:measurefilternonmeager}
	Let $\mu$ be a measure and let $\mathcal{F}$ be the filter of sets of measure $1$, then $F \subseteq \mathcal{P}(\omega)$ is non-meager.
\end{prop}
\begin{proof}
	Notice that the quotient $\mathcal{P}(\omega) / \mathcal{F}^*$ is c.c.c, where $\mathcal{F}^*$ is the dual ideal. So, if $\mathcal{F}$ is meager, then, by Talagrand's theorem (see \cite{Talagrand} or \cite[Theorem 4.1.2]{Bartoszynski}) there must
	be an interval partition $\omega = \bigcup_{n\in\omega} I_n$ such that, for all $B \in \mathcal{F}$ and almost all $n\in \omega$, we have that $B \cap I_n \neq \emptyset$. So, if $\mathcal{A}$ is an uncountable almost disjoint family, then $\{
	\bigcup_{a \in A} I_a : A \in \mathcal{A} \}$ is an uncountable antichain in $\mathcal{P}(\omega) / \mathcal{F}^*$.
\end{proof}

Below, we present several ways of producing concrete non-atomic measures on $\omega$.

\section{Constructions of measures}\label{Schemes}

As we have mentioned in the introduction, every ultrafilter can be seen as a measure. For an ultrafilter $\mathcal{U}$ define the Dirac delta $\delta_\mathcal{U}(A)$ to be $1$ if $A\in \mathcal{A}$, otherwise $0$.  Sometimes in what follow we will
just identify ultrafilters with its Dirac deltas. In this article we are more interested in non-atomic measures. By a result of Rudin \cite{Rudin}, every compact non-scattered space carries a non-atomic measure. Since $\beta\omega$ is one of such spaces, there is a non-atomic measure on $\omega$. 

\subsection{Extensions of density.}\label{sec:density}

Recall that the asymptotic density is defined as
\[ d(A) = \lim_{n\to\infty} \frac{|A\cap n|}{n} \]
provided the limit exists. The asymptotic density behaves very much like a measure, it is just undefined for many subsets of $\omega$. The simplest way of extending it to a measure is by using an ultrafilter limit.

Let $\mathcal{U}$ be an ultrafilter on $\omega$. Define the function $d_\mathcal{U} \colon \mathcal{P}(\omega) \to [0,1]$, an \emph{ultrafilter density}, by
\[ d_\mathcal{U}(A) = \lim_{n\to\mathcal{U}} \frac{|A \cap n|}{n}.\]

It is straightforward to check that $d_\mathcal{U}$  is a measure on $\omega$ which vanishes on points (provided $\mathcal{U}$ is free) and which is non-atomic.
It \emph{extends density}, i.e. $d_\mathcal{U}(A) = d(A)$ when the latter is defined. 

This approach can be generalized. For a measure $\mu$ on $\omega$ we can define
\[ d_\mu(A) = \int \frac{|A \cap n|}{n} \ d\mu. \]
Notice that for an ultrafilter $\mathcal{U}$ we have $d_{\delta_\mathcal{U}} = d_\mathcal{U}$.
So, we may define a measure $d_\mathcal{U}$, then $d_{d_\mathcal{U}}$ and so on. However, some measures defined as above cannot be achieved in this process: 

\begin{prop}\label{fix} There is a measure $\mu$ such that $\mu = d_\mu$.
\end{prop}

\begin{proof} Consider the mapping $d\colon \mathcal{M}_1 \to \mathcal{M}_1$ defined by $d(\mu) = d_\mu$. Notice that this mapping is continuous. Indeed, fix $\mu$, $\varepsilon>0$ and  $A\subseteq \omega$ and let $f$ be the sequence $(|A\cap n|/n)$.
	Take $\mu' \in d^{-1}[K(\varepsilon, \mu, A)]$. Then 
	\[ \left|\int f \ d\mu - \int f \ d\mu' \right| : = \delta < \varepsilon. \]
	Hence, $K(\varepsilon-\delta, \mu', f) \subseteq  d^{-1}[K(\varepsilon, \mu, A)]$ and so $d$ is continuous.

Since $\mathcal{M}_1$ is convex and compact (by Banach-Alaoglu theorem), we can use Schauder-Tychonoff fixed point theorem to conclude that there is $\mu$ such that $\mu=d(\mu)=d_\mu$.
\end{proof}

We finish this subsection by recalling the following well-known fact.

\begin{prop} If $\mu$ extends the density, then $\mu$, as a function from the Cantor set to $[0,1]$, is Lebesgue measurable. \end{prop}
	\begin{proof} By Borel's normal number theorem $\lambda(\{A\colon d(A)=1/2\}) = 1$.
	\end{proof} 

\subsection{Transfinite constructions \`a la Sikorski}\label{sec:transfinite}

The easiest way to obtain certain types of ultrafilters is to perform a construction under Continuum Hypothesis using transfinite induction of length $\omega_1$. As we will see, a similar situation holds for measures. In fact, we will not construct measures, but rather construct an embedding of $\mathcal{P}(\omega)$ into measure algebras (with measure $\lambda$), as if $\phi$ is such embedding, then $\mu(A) = \lambda(\varphi(A))$ is a measure on $\omega$. 

Let $\mathbb{C}$ be a complete Boolean algebra and let $\mathbb{A} \subseteq \mathcal{P}(\omega)$ be a Boolean algebra. For a Boolean homomorphism $\phi \colon \mathbb{A} \to \mathbb{C}$ define 
\[
	\phi^*(X) = \bigwedge \{ \phi(A) \colon A \in \mathbb{A}, \,\, X \subseteq A \}
\]
and
\[
	\phi_*(X) = \bigvee \{ \phi(A) \colon A \in \mathbb{A}, \,\, A \subseteq X \}
\]
for $X \in \mathcal{P}(\omega)$. By $\mathrm{alg}(\mathcal{A})$ we denote the Boolean algebra generated by $\mathcal{A}$ and let $\mathbb{A}(X) = \mathrm{alg}(\mathbb{A} \cup \{X\})$. 
\medskip

The following lemma will let us create the measure `step by step', by transfinite induction. The proof can be found in \cite{sikorski}.  

\begin{lem}\label{lem:Sikorski_extension}
    (\cite{sikorski}).
	Let $\mathbb{A}\subseteq \mathcal{P}(\omega)$ be a Boolean algebra that is closed under finite
	modifications, let $\mathbb{C}$ be a complete Boolean algebra and $\phi \colon \mathbb{A} \to \mathbb{C}$ be a homomorphism. If $X \subseteq \omega$ and $C \in \mathbb{C}$ is such that
    \[
        \phi_*(X) \subseteq C \subseteq \phi^*(X),
    \]
	then $\phi$ can be extended to a homomorphism $\phi' \colon \mathbb{A}(X) \to \mathbb{C}$ such that $\phi'(X) = C$. 
\end{lem}

We will show now the simplest application of the above for the purpose of constructing a measure on $\omega$. Recall that the measure algebra $\mathbb{B}$ is complete.

\begin{prop}\label{simplest} There is a Boolean epimorphism $\phi\colon \mathcal{P}(\omega) \to \mathbb{B}$.
\end{prop}

The usual ZFC proof of Proposition \ref{simplest} is to start with an independent family $\mathcal{A}$ of subsets of $\omega$, of size $\mathfrak{c}$. Define a bijection $f\colon \mathcal{A} \to \mathbb{B}$ which, by the fact that $\mathcal{A}$ is independent, can be 
automatically extended to a homomorphism $\phi\colon \mathrm{alg}(\mathcal{A}) \to \mathbb{B}$ and then extend $\phi$ to a homomorphism $\phi\colon \mathcal{P}(\omega) \to \mathbb{B}$.

However, under CH, we have more flexibility: Start with a trivial homomorphism $\phi_0\colon \{0,\omega\} \to \mathbb{B}$ and enumerate $ \mathbb{B} =
\langle B_\alpha\rangle_{\alpha<\omega_1}$. Then, for $\alpha<\omega_1$, use Lemma \ref{lem:Sikorski_extension} to extend $\phi_\alpha$ to $\phi_{\alpha+1}$ in such a way that $B_\alpha$ is in the range of $\phi_{\alpha+1}$. To see that this is possible  let $\mathcal{L} = \{X\subseteq \omega\colon \phi(X)\subseteq B_\alpha\}$, $\mathcal{R} = \{X\subseteq \omega\colon B_\alpha \subseteq \phi(X)\}$. Both $\mathcal{L}$ and $\mathcal{R}$ are countable and for
each $X\in \mathcal{L}$, $Y\in \mathcal{R}$ we have $X\subseteq Y$. So, there is $Z\subseteq \omega$ such that $X \subseteq^* Z \subseteq^* Y$ for each $X\in \mathcal{L}$, $Y\in \mathcal{R}$, since $(\omega,
\omega)$-gaps do not exists (see \cite{Scheepers}). At a limit step $\gamma$, for $\phi_\gamma$ we take the union of all the homomorphisms defined so far. 

The advantage of the proof using transfinite induction is that it can be easily modified to prove existence of homomorphisms with some additional properties. For example, the following was proved in \cite{Kunen}:  

\begin{prop}\label{prop:ker} (CH) There is a Boolean epimorphism $\phi\colon \mathcal{P}(\omega) \to \mathbb{B}$ whose kernel is a P-ideal.
\end{prop}

To prove it one has to add another 'promise' in the construction described above: for every countable $\mathcal{A}$ such that $\phi(A) = 1$ for each $A\in \mathcal{A}$, there is a pseudo-intersection $P$ of $\mathcal{A}$ such that $\phi(P)=1$. It is
not hard to see (and can be found in \cite{Kunen}) that this promise can be satisfied along the transfinite construction described above and that the resulting homomorphism $\phi$ will have the desired property. We will use this scheme, modifying some of these promises.

\subsection{Generic measure.}\label{sec:generic}

First, recall a standard fact about extending measures on Boolean algebras (which can be seen as a corollary of Lemma \ref{lem:Sikorski_extension}, see also e.g. \cite{Plachky}). \begin{lem}\label{lem:sikorski-measure}
	Suppose that $\mu$ is a measure defined on a Boolean algebra $\mathbb{A}\subseteq \mathcal{P}(\omega)$ closed under finite modifications, and that $X\subseteq \omega$. Let $r\in [0,1]$ be such that $\mu_*(X)\leq r \leq \mu^*(X)$. Then $\mu$ can
	be extended to a measure $\nu$ on $\mathbb{A}(X)$ such that $\nu(X)=r$.
\end{lem}

We present a natural example of a generic measure. Recall that the forcing $\mathcal{P}(\omega)/\mathrm{Fin}$ (with the reverse inclusion ordering) adds generically an ultrafilter. The following is its measure counterpart.

\begin{exa}\label{exa:generic_measure}
We  consider the family of all (free) \emph{partial measures} on $\omega$:
	\[ \mathbb{P} = \{\mu\colon \mu\mbox{ is a free measure on a countable Boolean algebra }\mathbb{A}\subseteq \mathcal{P}(\omega)\} \]
	with the ordering given by $\mu\leq \nu$ if $\mu \supseteq \nu$ (so, if the domain of $\nu$ is a subalgebra of the domain of $\mu$ and the measures agree on its elements).

	Then $\mathbb{P}$ is a forcing notion which is $\sigma$-closed: for every decreasing sequence $\langle \mu_n\rangle_n$, there is $\mu\in \mathbb{P}$ such that $\mu\leq \mu_n$ for each $n$.
	The forcing $\mathbb{P}$ adds generically a measure $\dot{\mu}$ on $\omega$ (as for every $X\subseteq \omega$ and a countable Boolean algebra $\mathbb{A}\subseteq \mathcal{P}(\omega)$ we can always extend $\mu$ to $\mathbb{A}(X)$, using  Lemma
	\ref{lem:sikorski-measure}). 

	Also, $1\forces{ \dot{\mu} \text{ is non-atomic}}{}$: Fix $N\in \omega, N>0$ and $\nu$ defined on a countable $\mathbb{A}\subseteq \mathcal{P}(\omega)$. One can build inductively a partition $\langle P_i\rangle_{i< N}$ of $\omega$ such that no $P_i$ contains an
	infinite element of $\mathbb{A}$. Then, $\nu_*(P_i)=0$ and $\nu^*(P_i)=1$ for every $i < N$. Hence, using Lemma \ref{lem:sikorski-measure} we can extend $\nu$ to $\nu'$ defined on $\mathbb{A}(P_0, P_1, \dots, P_{N-1})$ in such a way that $\nu'(P_i)=1/N$ 	for each $i$.

\end{exa}
	
\subsection{Solovay measures.}\label{sec:solovay}

In \cite{PbnDamian} the authors dust off an old construction due to Solovay (see \cite{Solovay}) of measures in the random model. We will sketch the construction (for the details, see \cite{PbnDamian}). 

Fix an ultrafilter $\mathcal{U}$ on $\omega$. 
We construct a $\mathbb{B}$-name for a measure on $\omega$ generated by $\mathcal{U}$.

Let $\dot{M}$ be an $\mathbb{B}$-name for a subset of $\omega$. Define the function $\mu_{\dot{M}}\colon\mathbb{B}\to[0,1]$ as follows:
\[ \mu_{\dot{M}}(B) = \lim_{k\to \mathcal{U}} \lambda \big(\llbracket k\in \dot{M} \rrbracket \wedge B\big). \]
It is immediate that $\mu_{\dot{M}}$ is the name of a finitely additive measure on $\mathbb{B}$. Also, for each $\mathbb{B}$-name $\dot{M}$ for a subset of $\omega$, the measure $\mu_{\dot{M}}$ is $\sigma$-additive and $\mu_{\dot{M}}$ is absolutely continuous with respect to $\lambda$. Therefore, there is a Radon-Nikodym derivative of $\mu_{\dot{M}}$, i.e. a measurable function $f_{\dot{M}}$ such that $\mu_{\dot{M}}(A) = \int_A f_{\dot{M}} \ d\lambda$.

Now let $\dot{r}_{\dot{M}}$ be the value of $f_{\dot{M}}$ on the generic real. Finally, let $\dot{\mu}_\mathcal{U}$ be a $\mathbb{B}$-name for a function
$\mathcal{P}(\omega)\to[0,1]$ such that for every $\mathbb{B}$-name $\dot{M}$ for a subset of $\omega$ we have
\[\forces{\dot{\mu}_\mathcal{U}(\dot{M}) = \dot{r}(\mu_{\dot{M}}).}{\mathbb{M}_\kappa} \]

In  \cite{Solovay} it is proved that 
 \[ \forces{\dot{\mu}_\mathcal{U} \mbox{ is a finitely additive probability measure on }\omega}{\mathbb{B}}. \]
Also, in \cite{PbnDamian} it is proved that  $\Vdash_{\mathbb{B}}``\dot{\mu}_\mathcal{U} \mbox{ is non-atomic }".$ 

We call $\dot{\mu}_\mathcal{U}$ a \emph{generic Solovay measure (associated to $\mathcal{U}$)}.

\subsection{Transfinite constructions \`a la Mokobodzki}\label{sec:Mokobodzki}

Under CH (or, more generally, under the assumption that $ \mathrm{cov}(\mathcal{M}) = \mathfrak{c}$) there is a measure $\mu$ on $\omega$ which is universally measurable, i.e. it is such that it is $\rho$-measurable as a function $\mu \colon 2^\omega
\to [0,1]$ for every Borel measure $\rho$ on the Cantor set. Such a measure is called a medial measure or Mokobodzki measure and it was constructed by Mokobodzki (\cite{Meyer}) and independently by Christensen (\cite{Christensen}) (see also
\cite{Godefroy}, \cite[538S]{Fremlin5}). 

We will present the idea of the construction and in Section \ref{sec:minimality} we will use it to define a peculiar example of a measure on $\omega$.

It will be convenient for us to change our setting and to use functionals rather than measures (see Section \ref{sec:preliminaries}). 
We will recursively construct sequences of functionals on $\ell^\infty$ with finite support. Note that any uniformly bounded sequence of functionals $\langle c_n \rangle_n$ can be treated as a linear function from $\ell^\infty$ to $\ell^\infty$ by putting \[c(x)(n)=c_n(x).\]
Define $e_n:\ell^\infty \to \RR$ by $e_n(x)=x(n)$.

\begin{df}
	Let $\langle c_n \rangle_n,\langle d_n\rangle_n$ be two sequences of bounded functionals on $\ell^\infty$. We say that $\langle c_n \rangle_n \preceq \langle d_n \rangle_n$ whenever for all $n$ we have $c_n\in conv(d_m:m\geqslant n)$, where
	$conv(X)$ is the convex hull of $X$.  Furthermore, let $\langle c_n \rangle_n \preceq^* \langle d_n \rangle_n$ whenever for all but finitely many $n$ we have $c_n\in conv(d_m:m\geqslant n)$. Both $\preceq$ and $\preceq^*$ are preorders.
\end{df}

\begin{df}
	A functional $c:\ell^\infty\to \RR$ has \emph{finite support} if it is a convex combination of $e_n$'s, i.e. it is of the form
	\[ c = \sum_{n\in F} a_n e_n, \]
	where each $a_n$ is non-negative and $\sum_{n\in F} a_n = 1$. We will call the set $F$ the \emph{support} of $c$.
\end{df}

Observe that, if $\langle c_n\rangle_n \preceq \langle e_n \rangle_n$, then each $c_n$ has finite support. Additionally, for any $r\in \RR$,  $c_n(\vec{r})=r$, where $\vec{r}$ is the sequence constantly equal $r$. The following fact is the hearth of the construction.
\begin{fact}\label{coherence}
	If $\langle c_n\rangle_n \preceq^* \langle d_n\rangle_n$ and the sequence $(d_n(x))_n$ converges, then \[ \lim_n c_n(x)=\lim_n d_n(x).\]
\end{fact}

The Mokobodzki construction is performed under CH, by first defining a sequence $\langle \langle c_n^\alpha\rangle_n:\alpha\in \omega_1\rangle$ which is $\preceq^*$-increasing. Along the way, the limits of $c^\alpha(x)$ exist for more and more $x\in \ell^\infty$. Finally we define a functional by $\varphi(x)=\lim_n c_n^\alpha(x)$ for the smallest $\alpha$ where such limit exists.

As an illustration of the above, we can perform the simplest construction of this form. Assume CH and enumerate $\ell^\infty = \{x^\alpha\colon \alpha<\omega_1\}$. Start with $c_n^0 = e_n$ and suppose we have defined $\langle c_n^\alpha\rangle_n$. There is a
sequence $k_n$ such that $\lim c^\alpha_{k_n}(x_n)$ does exist. Define $c_n^{\alpha+1} = c_{k_n}^\alpha$ for each $n$. If $\alpha$ is a limit ordinal, then fix a cofinal sequence $(\alpha_n)_n$ and find $(k_n)$ such that $\lim
c^{\alpha_n}_{k_n}(x^\alpha_n)$ does exist and proceed as in the successor case.
Clearly $\langle c^\alpha_n\rangle_n$ is an $\preceq^*$-increasing sequence, with respect to $\alpha$, and for each $x\in \ell^\infty$ there is $\alpha$ such that
$\lim c^\alpha_n (x)$ does exist. Finally, we obtain a functional, $\varphi(x)=\lim_n c_n^\alpha(x)$ for the smallest $\alpha$ where such limit exists (see \ref{coherence}). 

The above construction is in fact nothing else but a limit along an ultrafilter (when we were picking subsequences, we actually chose subsets in an ultrafilter). So, the measure obtained in this way is just the Dirac delta of an ultrafilter. If instead we start with
$c_n^0 = (e_1 + \dots + e_n)/n$ and proceed in the above way, then we would get a measure of the form $\mu(A) = \lim_{n\to \mathcal{U}} |A\cap n|/n$. 

The construction of Mokobodzki imposes more conditions on the sequences $\langle c^\alpha_n\rangle_n$ to obtain universal measurability of the resulting functional. In Section \ref{sec:minimality} we will add yet another conditions to obtain a measure
which is minimal with respect to the Rudin-Blass ordering and which is not a Q$^+$-measure.

\section{Rudin-Blass and Rudin-Keisler orderings on measures}\label{sec:pqs}

We start with a definition inspired by the well-known notion of Rudin-Blass ordering on filters (see \cite{Laflamme-RB}). By $f[\mu]$ we mean the measure $\nu$ such that $\nu(A) = \mu (f^{-1}[A])$ for each $A$.

\begin{df}\label{def:rudin_blass_ordering}
	Let $\mu, \nu$ be measures on $\omega$. We say that $\nu$ is \emph{Rudin--Blass reducible} to $\mu$ (denoted by $\nu \leq_{RB} \mu$), if there is a finite-to-1 function $f \colon \omega \to \omega$ such that $f[\mu] = \nu$. 
\end{df}

\begin{df}\label{def:nearly_Dirac}
	We say that a measure $\mu$ is \emph{nearly ultra} if there exists an ultrafilter $\mathcal{U}$ such that $\delta_{\mathcal{U}} \leq_{RB} \mu$. 
\end{df} 

It is known that every ultrafilter extension of the asymptotic density is nearly ultra. Also, consistently (under Filter Dichotomy) every measure is nearly ultra. However, consistently there is
a measure which is not nearly ultra (for the proofs of all these statements, see \cite{cancino}). 
\vspace*{.5em}

We say that a measure $\mu$ is Rudin-Blass minimal (RB-minimal, in short) if and only if for any $\nu$ such that $\nu \leq_{RB} \mu$ we have $\mu \leq_{RB} \nu$. 
It is known that an ultrafilter is a Q-point if and only if it is minimal in the Rudin--Blass ordering for ultrafilters. 

\vspace*{.5em}

\begin{df}\label{def:rudin_keisler_ordering}
	Let $\mu, \nu$ be measures on $\omega$. We say that $\nu$ is \emph{Rudin--Keisler reducible} to $\mu$ (denoted by $\nu \leq_{RK} \mu$) if there is a function $f \colon \omega \to \omega$ such that $f[\mu] = \nu$.
\end{df}

So, Rudin-Blass ordering is just Rudin-Keisler ordering with the 'finite-to-one' requirement. Again, the Rudin-Keisler ordering was investigated for ultrafilters and it is known that being a selective ultrafilter is equivalent to being minimal in this ordering (see \cite{Blass-thesis}). 

The structure of the both Rudin-Blass and Rudin-Keisler orderings on ultrafilters has been extensively studied (see e.g. \cite{Verner-Dilip}, \cite{Borisa-Dilip}).

\begin{prop}\label{ksi} Suppose that $f\colon \omega \to \omega$ and $\mu$ are such that $f[\mu] = \mu$. Then for every  $N\subseteq \omega$ we have $\mu(f[N]) \geq \mu(N)$.
\end{prop}
\begin{proof} Since $N \subseteq f^{-1}[f[N]]$ and $\mu(f^{-1}[A]) = \mu(A)$ for every $A$, the conclusion follows.
\end{proof}

\begin{prop} Suppose that $\mu$ is a measure on $\omega$. Then, there is a finite-to-one function $g\colon \omega \to \omega$ such that $g[\mu] = g[d_\mu]$ and, consequently, there is $\rho \leq_{RB} \mu, d_\mu$.
\end{prop}
\begin{proof}
	Let $f\colon \omega \to \omega$ be a function such that $\lim_{n} f(n+1)/f(n) = \infty$ and let $I_n = \big[f(n), f(n+1)\big)$ for each $n$.

	By the fact that $\{A\colon d_\mu(A)=0\}$ is non-meager, we may find an infinite $N\subseteq \omega$ such that $d_\mu\big(\bigcup_{k\in N} I_k\big)=0$. Let $\langle F_n\rangle_n$ be a sequence of intervals such that $N = \bigcup_n F_n$ and $\max F_n + 1 < \min F_{n+1}$ for
	every $n$. Finally, define $J_n =\big[f(\min F_n), f(\min F_{n+1}) \big)$ and $K_n = \big[f(\max F_n), f(\min F_{n+1}) \big)$.
	\medskip

	\textbf{Claim.} For each $A\subseteq \omega$ we have $\mu\big(\bigcup_A J_i\big) = d_\mu\big(\bigcup_A J_i\big)$.
	\medskip

	Fix $A\subseteq \omega$. First, notice that we have $\mu( \bigcup_A J_i ) = \mu (\bigcup_A K_i)$ and so also $d_\mu( \bigcup_A J_i ) = d_\mu (\bigcup_A K_i)$. Denote $B = \bigcup_A K_i$.
	For each $\varepsilon>0$ there is $m$ such that for every $n>m$ we have
	\[ \left|\frac{|B \cap n|}{n} - \chi_{B}(n)\right| < \varepsilon. \]
By the fact that $\mu$ vanishes on points, it means that
	\[ d_\mu\Big( \bigcup_A J_i\Big) = d_\mu(B)  = \int \frac{|B \cap n|}{n} \ d\mu = \int \chi_B \ d\mu = \mu(B) = \mu\Big(\bigcup_A J_i\Big). \]

	To conclude the proof of the proposition, define $g\colon \omega \to \omega$ by $g(k) = i$ if $k\in J_i$.
\end{proof}

\begin{cor} For every ultrafilter $\mathcal{U}$, there is an ultrafilter $\mathcal{V}$ such that $\mathcal{V} \leq_{RB} d_\mathcal{U}, \mathcal{U}$. 
\end{cor}

\section{Special kinds of measures}\label{sec:qmeasure}

\subsection{P-measures}\label{sec:p_measures}

In this section we present the notion of P-measure, the natural generalization of P-points to the measure context. 
It was studied by several people, starting with Mekler \cite{mekler} (as AP measures, or measures with the additive property), Blass, Frankiewicz, Plebanek and Ryll-Nardzewski \cite{blassple}, Grebik \cite{grebik}, Kunisada \cite{Kunisada}. 
Here we present the definitions and basic facts about them. 

First, recall the definition of a P-filter.

\begin{df}\label{df:p_filter}
	We say that a filter $\mathcal{F}$ is a \emph{P-filter} if for each family $\langle A_n\rangle_n$ such that $A_n \in \F$ for all $n \in \omega$, there is a set $X \in \F$ which is a pseudo-intersection of $\langle A_n\rangle_n$. We say that a filter $\mathcal{U}$
	is a P-point if it is an ultrafilter and a P-filter.
\end{df}

The following definition appeared in \cite{mekler}:

\begin{df} A measure $\mu$ is \emph{a P-measure} if for every descending sequence $\langle A_n\rangle_n$  there is a pseudo-intersection $P$ of $\langle A_n\rangle_n$ such that $\mu(P) = \lim \mu(A_n)$. 
\end{df}

The above definition can be rewritten equivalently in the language of partitions (which will be slightly more convenient for our purposes). We say that a set $S$ is a \emph{pseudo-selector} of a partition $\langle A_n\rangle_n$  of $\omega$ if $S\cap A_n$ is finite for each $n$.

\begin{prop} A measure $\mu$ is a P-measure if and only if for every partition $\langle A_n\rangle_n$ there is a pseudo-selector $S$ of $\langle A_n\rangle_n$ such that \[ \mu(S) = 1 - \sum_{n} \mu(A_n). \]
\end{prop}

Notice that an ultrafilter $\mathcal{U}$ is a P-point if and only if $\delta_\mathcal{U}$ is a P-measure. The standard example of a P-measure which is non-atomic is this given in Section \ref{sec:density}: 

\begin{thm}\label{thm:density} \cite{mekler}
Let $\mathcal{U}$ be a P-point. Then the ultrafilter density
\[ d_\mathcal{U}(A) = \lim_{n\to \mathcal{U}} \frac{|A\cap n|}{n} \]
is a non-atomic P-measure on $\omega$. 
\end{thm}

In \cite{cancino} the authors gave a fundamentally different example of a P-measure. Namely, they proved that the measure induced by the homomorphism from Proposition \ref{prop:ker} is a P-measure. It follows from the fact that the filter of measure $1$ sets is a P-filter and that the Lebesgue measure $\lambda$ is
$\sigma$-additive on $\mathbb{B}$.

The main result of \cite{mekler} was that consistently there are no P-measures. Mekler asked if the existence of P-measures imply the existence of P-points. The problem was solved recently in \cite{cancino}:

\begin{thm}\cite{cancino}
Consistently, there is a P-measure but there are no P-points.
\end{thm}

In \cite{mekler} Mekler introduced another definition (under the name of AP0 measures):

\begin{df} A measure $\mu$ is a \emph{P$_0$-measure} if for every descending sequence $\langle A_n\rangle_n$ of measure 1, there is a pseudo-intersection $P$ of $\langle A_n\rangle_n$ such that $\mu(P) = \lim \mu(A_n) (=1)$. 
\end{df}

An ultrafilter $\mathcal{U}$ is a P-point if and only if $\delta_\mathcal{U}$ is a P-measure, if and only if $\delta_\mathcal{U}$ is a P$_0$-measure. So, in fact P$_0$-measures are also possible generalizations of P-points to the context of
measures.
Notice that $\mu$ is a P$_0$-measure if and only if the filter $\{F\colon \mu(F)=1\}$ is a P-filter. So, the property of being a P$_0$-measure is more a property of the filter of measure 1 sets.

\begin{prop}\label{prop:p_stronger_than_p0}
(CH) There exists a P$_0$-measure that is not a  P-measure.
\end{prop}

\begin{proof}
	Notice that if $\mu$ is a P-measure, then it is $\sigma$-additive on the Boolean algebra $\mathcal{P}(\omega)_{/\mu = 0}$. Consider the Cohen algebra $\mathbb{C} = \text{Bor}(2^{\omega})_{/\text{Meager}}$. It is a complete Boolean algebra. Moreover there exists a strictly positive measure $\lambda'$ on $\mathbb{C}$ and this measure is not $\sigma$-additive (see \cite{Mirna} for more details). Hence we can create a measure $\mu'$ on $\omega$ by transfinite
	induction (like in Proposition \ref{prop:ker}) which is a P$_0$-measure and it is not $\sigma$-additive on $\mathcal{P}(\omega)_{/\mu = 0}$, thus it is not a P-measure. 
\end{proof}

Now we note the relation of notions of P and P$_0$-measures with Rudin-Blass and Rudin-Keisler orderings. Namely, below a P$_0$-measure, the order $\leqslant_{RK}$ collapses to $\leqslant_{RB}$.
\begin{fact}\label{fact:collbelowp}
	If $\mu$ is a P$_0$-measure and $\nu\leqslant_{RK}\mu$, then $\nu\leqslant_{RB} \mu$. Furthermore the properties of being a P-measure and P$_0$-measure are both $RB$ and $RK$-hereditary.
\end{fact}
\begin{proof}
	To prove the former, take $f$ witnessing $\nu\leqslant_{RK}\mu$. Then, for each $n$, we have $\mu(f^{-1}[\{n\}])=0$. Using the fact that $\mu$ is P$_0$, pick $S$, a full measure pseudo-selector for $\big\langle f^{-1}[\{n\}]\big\rangle_n$. Then we can change values of $f$ on $\omega\setminus S$ making the function finite-to-one, but not changing $f[\mu]$.
	
	Now we prove that the property of being a P-measure (P$_0$-measure) is $RB$-hereditary. 
    Given $f[\mu]$ and a partition (of measure 0 sets) $\langle A_n\rangle_n$ apply the P (respectively P$_0$) property of $\mu$ to $\langle f^{-1}[A_n]\rangle_n$ to obtain $S$. Then $f[S]$ is a desired pseudo-selector as $f[\mu](f[S])\geqslant \mu(S)=1-\sum_n f[\mu](A_n)$.
\end{proof}
Before we move to Q-measures, we give a result on the generic existence of P-measures. This is an analog of a known result \cite{Ketonen1976}, that every filter base of size less than $\mathfrak{c}$ can be extended to a
P-point if and only if $\mathfrak{d}=\mathfrak{c}$.
\begin{thm}\label{d=c}
	$\mathfrak{d} = \mathfrak{c}$ if and only if every measure $\mu$ defined on a Boolean algebra $\mathbb{B}\subseteq \mathcal{P}(\omega)$ such that $|\mathbb{B}| < \mathfrak{c}$, can be extended to a P-measure on $\omega$.
\end{thm}
\begin{proof}
	($\Rightarrow$) Let $\mu$ be a measure defined on $\mathbb{B} \subseteq \mathcal{P}(\omega)$ such that $|\mathbb{B}| < \mathfrak{c}$. Let $\{X_\alpha : \alpha \in \mathfrak{c}\} = [\omega]^\omega$ and let $\{ \langle A^\alpha_n\rangle_{n} : \alpha \in \mathfrak{c} \}$ be an enumeration of all partitions of $\omega$, where each one is repeated $\mathfrak{c}$-many times. Recursively, we are going to construct a sequence $\{ \mathbb{B}_\alpha : \alpha \in \mathfrak{c}\}$ of Boolean algebras and a sequence $\{\mu_\alpha : \alpha \in \mathfrak{c}\}$ of measures such that
	\begin{enumerate}
		\item $\mathbb{B}_0 = \mathbb{B}$ and the sequences $\{ \mathbb{B}_\alpha : \alpha \in \mathfrak{c}\}$, $\{\mu_\alpha : \alpha \in \mathfrak{c}\}$ are $\subseteq$-increasing,
		\item $\mu_\alpha$ is a measure on $\mathbb{B}_\alpha$,
		\item $|\mathbb{B}_\alpha| = \max\{|\mathbb{B}|, |\alpha| \}$,
		\item $X_\alpha \in \mathbb{B}_\alpha$,
		\item if $\langle A^\alpha_n\rangle_{n} \subseteq \mathbb{B}_\alpha$, then there is a pseudo-selector $S \in \mathbb{B}_\alpha$ for that partition such that $\mu_\alpha (S) = 1 - \sum_n \mu_\alpha (A^\alpha_n)$.
	\end{enumerate}
	It is clear that if we are able to carry out the construction, then the conclusion follows. Now, assume that the construction can be carried out for all $\beta < \alpha$. Let $\mathbb{B}' = \bigcup_{\beta < \alpha}\mathbb{B}_\beta$ and $\mu' =
	\bigcup_{\beta < \alpha} \mu_\beta$. Then $\mu'$ and $\mathbb{B}'$ satisfies (1), (2) and (3). We can extend them to satisfy (4) by using Lemma \ref{lem:Sikorski_extension}, so we will assume that (4) is also satisfied. Let $\mathcal{B} = \{ b
	\in \mathbb{B}': \exists^\infty n \in \omega (A^\alpha_n \nsubseteq b) \}$. For each $b \in \mathcal{B}$, find a strictly increasing $f_b : \omega \rightarrow \omega$ such that $f_b(0) = 0$, and for every $k \in \omega$,  we have
	\[ b^c \cap (\bigcup\{A^\alpha_i \cap f_b(k+1) : i \in (f_b(k),f_b(k+1)]\})\neq \emptyset.\] Since $|\mathcal{B}| < \mathfrak{d},$ there must be an increasing $g:\omega \rightarrow \omega$ such that, for each $b \in \mathcal{B}$, $g \nleq^* f_b$.
	Let $S = \bigcup_{n \in \omega} g(n) \cap A^\alpha_n$ and note that $S$ is a pseudo-selector of the partition $\langle A^\alpha_n\rangle_{n}$. Observe that, if $b \in \mathcal{B}$, then there are infinitely many $i>0$ such that $g(i) > f_b(i)$. For such $i$ we have that, if $j \in (f_b(i-1),f_b(i)]$, then $j \geq i$, and therefore $g(j) \geq g(i) > f_b(i)$, which means that $S$ avoids infinitely many elements of $b$, so $S \nsubseteq b$. So, for each $b \in \mathbb{B}'$ such that $S \subseteq b$ we have that $b \notin \mathcal{B}$ and therefore there is an $N \in \omega$ such that
   \[ 
    1 = \mu'(\omega) \leq \mu'(b) + \sum_{i <N}\mu'(A^\alpha_i) 
\]
    which means that $\mu'(b) \geq 1 - \sum_{n<N}\mu'(A^\alpha_n)$. As a conclusion we have $\mu'^*(S) = 1 - \sum_n\mu'(A^\alpha_n)$, so Lemma \ref{lem:Sikorski_extension} provides a way to extend $\mu'$ and $\mathbb{B}'$ so all the conditions are satisfied.
	
	($\Leftarrow$) Assume that $\mathfrak{d} < \mathfrak{c}$. Let $\{ f_\alpha : \alpha \in \mathfrak{d}\}$ be a dominating family, and let $D_\alpha = \{ \langle n,m \rangle \in \omega \times \omega : m \leq f_\alpha(n) \}$ and $C_n = \{n\} \times \omega$. Let $\mathbb{B}$ be the Boolean algebra generated by the $D_\alpha$'s and the $C_n$'s and let $\mu$ be a measure defined on $\mathbb{B}$ 
	 by
	$$
	\mu (A) = \begin{cases}
		0 & \text{if }\exists N \in \omega \ (\forall n> N (A \cap C_n \text{ is finite})), \\
		1 & \text{otherwise.}
	\end{cases}
	$$
Let $X_n = (\omega \times \omega)\setminus \bigcup_{i<n} C_i$. Then $\langle X_n \rangle_{n \in \omega}$ is a decreasing sequence of sets of measure 1. However, notice that for every possible pseudo-intersection $P$, there must be an $\alpha$ such that $P \subseteq^* D_\alpha$, and since $\mu(D_\alpha)=0$, then $\mu^*(P) = 0$, so $\mu$ cannot be extended to a P-measure.
\end{proof}

\subsection{Q-measures}\label{sec:Q}

We will first recall the definition of a Q-point. A set $S$ is called a \emph{selector} of a partition $\langle P_n\rangle_n$ if $|S\cap P_n|\leqslant 1$ for each $n$.

\begin{df}\label{df:q_point}
	An ultrafilter $\mathcal{U}$ is a \emph{Q-point} if any partition $\langle P_n\rangle_n$ into finite sets has a selector in $\mathcal{U}$.
\end{df}

The natural extension of this definition for measures seems to be the following. 

\begin{df}
    \label{df:q_measure}
	We say that $\mu$ is a \emph{Q-measure} if any partition $\langle P_n\rangle_n$ into finite sets has a selector $S$ with $\mu(S)=1$.  
\end{df}

As with P$_0$-measures, this definition is a property of the filter of measure $1$ sets.

\begin{prop}\label{prop:q-filter} A measure $\mu$ is a Q-measure if and only if $\{F\subseteq \omega\colon \mu(F)=1\}$ is a Q-filter.
\end{prop}

On the other hand, the selector $S = \{s_0, s_1, \dots\}$ of full measure has the following property of preserving the measure: let $N \subseteq \omega$ be arbitrary. For $A = \bigcup_{n \in N} A_n$ we have 
    \[
		\mu\big(\{s_n : n \in N\}\big) = \mu(S \cap A) = \mu(A). 
    \]
    We will call such selector a \emph{measure preserving selector}. 

We will list some easy observations. 

\begin{prop} If $\mathcal{U}$ is a Q-point, then $\delta_\mathcal{U}$ is a Q-measure. 
\end{prop}

Contrary to the case of P-measures, the existence of Q-measures (even non-atomic) implies the existence of Q-points.

\begin{prop}\label{q-m-q-p} The existence of Q-measures imply the existence of Q-points. Therefore, consistently there are no Q-measures.
\end{prop}
\begin{proof}
	The first part follows from Proposition \ref{prop:q-filter} and the obvious fact that every ultrafilter extending a Q-filter is a Q-point. The second follow from the classical results on models without Q-points (see e.g. \cite{Miller}). 
	\end{proof}

In fact, we can prove a slightly stronger fact. Let us introduce a weaker notion of Q$^+$-measures. 
\begin{df}
	We say that $\mu$ is a \emph{Q$^+$-measure} if any partition $\langle P_n\rangle_n$ into finite sets has a selector $S$ with $\mu(S)>0$.  
\end{df}
Notice that a measure is Q$^+$ if and only if the filter of sets of full measure is a Q$^+$ filter. The following theorem was proved independently in \cite{Luis}.

\begin{thm} Consistently, there are no Q$^+$-measures. 
\end{thm}
\begin{proof} Assume Filter Dichotomy and suppose that there is a Q$^+$-measure $\mu$. Let $\mathcal{F} = \{F\colon \mu(F)=1\}$. Then, $\mathcal{F}$ is non-meager (Proposition \ref{prop:measurefilternonmeager}). By Filter Dichotomy, there is an ultrafilter $\mathcal{U} \leq_{RB}
	\mathcal{F}$. We will show that $\mathcal{U}$ is a Q-point. Indeed, let $f$ be the finite-to-one function witnessing $\mathcal{U}\leq_{RB} \mathcal{F}$ and let $g\colon \omega\to \omega$ be any finite-to-one function.  As $\mathcal{F}$ is a
	Q$^+$-filter, there is a selector $S$ for $\big\langle(g\circ f)^{-1}(n)\big\rangle_n$ such that $\mu(S)>0$. Then $S' = f[S]$ is a selector for $\langle g^{-1}(n)\rangle_n$ and $ S' \in \mathcal{U}$. Otherwise $\mu (f^{-1}[S'])=0$ which yields a contradiction as $S\subseteq
	f^{-1}[S']$. But Filter Dichotomy implies that there are no Q-points (see \cite[Corollary 15]{Blass-NCF}).
\end{proof}

Of course, the inevitable question is if it is possible to get a non-atomic Q-measure consistently. Recall that there was a canonical way of getting a non-atomic P-measure from a P-point by extending the asymptotic density, see Theorem \ref{thm:density}. However, in
this way we will never get a Q-measure, or even a Q$^+$-measure, as these cannot be shift invariant (see Proposition \ref{Q-not-shift}).

Below we present an example of a non-atomic Q-measure, which will be obtained by the scheme described in Section \ref{Schemes}.

\begin{lem}\label{lem:extension_partition_selector}
	Let $\mathbb{A}$ be a countable Boolean algebra on $\omega$ (closed under finite modifications), let $\mathbb{B}$ be a complete Boolean algebra, and let $\langle P_n \rangle$ be a partition of $\omega$ into finite sets. Let $\phi \colon
	\mathbb{A} \to \mathbb{B}$ be a homomorphism, such that $\phi(A) = \mathbf{0}$ for any finite $A \in \mathbb{A}$. There exists $S \subseteq \omega$, such that $|S \cap P_n| \leq 1$ for every $n \in \omega$ and 
    \[
		\phi^{*}(S) = \mathbf{1}.  
    \]
	Therefore, \(\phi\) can be extended to a homomorphism \(\phi' \colon \mathbb{A}(S) \to \mathbb{B} \) such that
    \[
        \phi'(S) = \mathbf{1}.  
    \]
\end{lem}

\begin{proof}
	It suffices to show that there exists $S$, such that $|S \cap P_n| \leq 1$ for every $n \in \omega$ and for any $A \in \mathbb{A}$, if $S \subseteq \mathbb{A}$, then $\phi(A) = \mathbf{1}$. 
    \vspace*{.5em}

	Let $\langle A_n\rangle_{n \in \omega}$ be a sequence of all sets from $\phi^{-1}[\mathbb{B} \setminus \{\mathbf{1}\}]$. Let $f \colon \omega \to \omega$ be a function such that $x \in P_{f(x)}$. 
    Now we will recursively construct $S$. We set 
    \[
        s_i \in A_i^c \setminus \Big( \bigcup_{j < i} P_{f(s_j)} \Big).
    \]
    Note that it is possible as the set on the right-hand side is non-empty. 

    \vspace*{.5em}

	Finally, set $S = \{s_i \colon i \in \omega\}$. It is easy to see that if $ S \subseteq A$, then $\phi(A) = \mathbf{1}$. 
\end{proof}

\begin{thm}\label{thm:qmeasure}
	Assume CH. There exists a non-atomic Q-measure. 
\end{thm}

\begin{proof}
	We proceed as in the schema presented in Section \ref{Schemes}. Namely, we construct a homomorphism $\phi
	\colon \mathcal{P}(\omega) \to \mathbb{B}$ which fulfills the following conditions:
    
    \begin{itemize}
		\item[(1)] for every $B \in \mathbb{B}$ there is $X \subseteq \omega$ such that $\phi(X) = B$,
		\item[(2)] for every $\langle P_n\rangle_{n \in \omega}$, a partition of $\omega$ into finite sets, there exists a selector $S$ of this partition, such that $\phi(S) = \mathbf{1}$.  
    \end{itemize}

    In this way we can obtain a measure $\mu$ defined as $\mu(X) = \lambda(\phi(X))$ for every $X \subseteq \omega$. It follows from (1) that $\mu$ is non-atomic and from (2) that $\mu$ is a Q-measure. 
    \vspace*{.5em}

	Let $(\mathcal{P}_\alpha)_{\alpha < \omega_1}$ be a sequence of all partitions of $\omega$ into finite sets. For the steps of the form $\gamma + 2k$, $\gamma$  - limit (or $0$), $k\in \omega$, we can use Lemma \ref{lem:extension_partition_selector} to
	find a selector $S$ of partition $\mathcal{P}_{\gamma + k}$ such that  $\phi^{*}_{\gamma + 2k}(S) = \mathbf{1}$ and set
	$\phi_{\gamma + 2k + 1}(S) = \mathbf{1}$. For the steps of the form $\gamma+2k+1$, we can take care of condition (1), in the way described in Section \ref{sec:transfinite}.

    Then $\phi = \bigcup_{\alpha < \omega_1} \phi_{\alpha}$ satisfies (1) and (2). 
\end{proof}

Recall that an ultrafilter is a Q-point if and only if it is RB-minimal. We will show that partially we have an analogous result for measures. 

\begin{prop}
    \label{thm:q_measures_minimal_rb}
	If $\mu$ is a Q-measure and $\nu \leq_{RB} \mu$, then $\mu \leq_{RB} \nu$. 
\end{prop}

\begin{proof}
	Let $f \colon \omega \to \omega$ be a finite-to-1 function witnessing $\nu \leq_{RB} \mu$. 
    Denote $A = f[\omega] = \{ a_1, a_2, \ldots \}$. 
    Then $\nu(A) = \mu(f^{-1}[A]) = \mu(\omega) = 1$. 

    For $i \in \omega$ let $A_i = f^{-1}[\{ a_i \}]$. 
	Obviously, $\langle A_i\rangle_i$ is a partition of $\omega$ into finite sets and $\mu(A_i) = \nu(\{ a_i \}) = 0$ for every $i$. 

	Let $S = \{s_i\}_{i \in \omega}$ be a measure preserving selector of $\langle A_i\rangle_i$ promised by the fact that $\mu$ is a Q-measure.
    Define $g_0 \colon A \to S$ by $g_0(a_i) = s_i$. Clearly, $g_0$ is a bijection with $f_{\upharpoonright S}$ being its inverse and can be extended to a finite-to-1 function $g \colon \omega \to \omega$, such that $g[\omega] = S$. 

    \vspace*{.5em}

	Finally, we show that $g$ witnesses that $\mu \leq_{RB} \nu$. 
    Let $X \subseteq \omega$. Then we have 
    \begin{align*}
        \nu(g^{-1}[X]) 
        &= \nu(g^{-1}[X \cap S]) 
        = \nu((g^{-1}[X \cap S] \cap A) \cup (g^{-1}[X \cap S] \cap A^c)) \\
        &= \nu(g^{-1}[X \cap S] \cap A) + \nu(g^{-1}[X \cap S] \cap A^c) \\
        &= \nu(g^{-1}[X \cap S] \cap A)
        = \nu(g_0^{-1}[X \cap S]) \\
        &= \nu(f[X \cap S])
        = \mu(f^{-1}[f[X \cap S]])
        = \mu((X \cap S) \cap S) \\
        &= \mu(X). 
    \end{align*}
\end{proof}

In the next section we will show that the above implication, in general, cannot be reversed.

We overview the main results about Q-measures contained in \cite{Luis}.

\begin{thm}(\cite{Luis}) There is no Q$^+$-measure in the Laver's model for no Q-points.
\end{thm}

In fact the authors proved a more general result saying that no $\sigma$-bounded-cc ideal is rapid in the above model.

Also, the authors prove a result analogous to the well-known fact for ultrafilters saying that if $\mathfrak{d} = \mathrm{cov}(\mathcal{M})$, then every filter generated by less than $\mathfrak{d}$ many sets can be extended to a Q-point. This generalizes Theorem \ref{thm:qmeasure}.

\begin{thm}(\cite{Luis})  The assumption $\mathfrak{d} = \mathrm{cov}(\mathcal{M})$ is equivalent to the statement: every non-atomic measure defined on a Boolean algebra of size less than $\mathfrak{d}$ many sets can be extended to a non atomic Q-measure.
\end{thm}

\subsection{Selective measures}\label{sec:selective}

Recall the definition of selective ultrafilter.

\begin{df}\label{def:selective_filter}
A filter $\mathcal{F}$ is \emph{selective} if for every $\langle A_n\rangle_n$, a 
    partition of $\omega$, such that $A_n^c \in \mathcal{F}$ for each $n$, there exists a selector $S$ of this partition, such that $S \in \mathcal{F}$.
\end{df}

Again, it seems that the natural generalization for measures is the following:

\begin{df}
    \label{def:selective_measure}
    A measure $\mu$ is \emph{selective} if any partition $\langle P_n\rangle_n$ has a selector $S$ such that $\mu(S) = 1 - \sum_n \mu(A_n)$. 
\end{df}

Note that the selector $S$ from the definition above is a measure preserving selector. 

It is a standard fact that the selective ultrafilters are exactly those which are P-points and Q-points. The analogous fact holds for measures.

\begin{prop}
    \label{lem:selective_p_q}
    A measure $\mu$ is selective if and only if it is a P-measure and a Q-measure. 
\end{prop}

\begin{proof}
    $(\Longrightarrow)$. Follows directly from the definitions. 

    $(\Longleftarrow)$. 
    Let $\langle A_n\rangle_n$ be a partition of $\omega$. 
    Let $S'$ be a pseudo-selector of this partition such that $\mu(S') = 1 - \sum_n \mu(A_n)$. 
    Consider the family $\langle S' \cap A_n\rangle_n$. 
    It can be easily extended to a partition of $\omega$ into finite sets (e.g. by adding the singletons of all the elements of $\omega \setminus \bigcup_{n} (S' \cap A_n)$). 
    Let $S_0$ be a selector of this partition such that $\mu(S_0) = 1$. 
    Finally, we obtain $S = S' \cap S_0$, a selector of $\langle A_n\rangle_n$, such that $\mu(S) = \mu(S') = 1 - \sum_n \mu(A_n)$. 
\end{proof}

Analogously to P$_0$-filters, we have the following notion.
\begin{df}
    \label{def:0selective_measure}
    A measure $\mu$ is \emph{0-selective} if for every $\langle A_n\rangle_n$, a partition of $\omega$ into sets of measure $0$, there exists a selector $S$ of this partition, such that $\mu(S) = 1$. 
\end{df}

 Clearly a measure is 0-selective if and only if it is Q and P$_0$ and all selective measures are 0-selective. We obtain the following theorem as a direct corollary of Proposition \ref{thm:q_measures_minimal_rb} and Proposition \ref{fact:collbelowp}.

\begin{thm}
	If $\mu$ is a 0-selective measure and $\nu \leq_{RK}  \mu$, then $\mu \leq_{RK} \nu$. 
\end{thm}

Note that a selective measure can be constructed under CH in the similar way to Q-measure (see Theorem \ref{thm:qmeasure}). Below we present some forcing constructions.

We  consider the forcing $\mathbb{P}$ from Example \ref{exa:generic_measure}:
	\[ \mathbb{P} = \{\mu\colon \mu\mbox{ is a free measure on a countable Boolean algebra }\mathbb{A}\subseteq \mathcal{P}(\omega)\} \]
and $\mu\leq \nu$ if $\mu \supseteq \nu$.
\begin{prop}\label{prop:generic_selective} Let $\dot{\mu}$ be the generic measure added by $\mathbb{P}$. Then
	 \[ \forces{ \dot{\mu} \mbox{ is selective.}}{\mathbb{P}}\]
\end{prop}
\begin{proof}
 We will use Proposition \ref{lem:selective_p_q}.

	First, we show that $\forces{\dot{\mu}\text{ is a Q-measure}}{\mathbb{P}}$. Indeed, let $\langle P_n\rangle_n$ be a partition of $\omega$ into finite sets and let $\mu\colon \mathbb{A} \to [0,1]$ be an element of $\mathbb{P}$. We claim that there is a selector $S$ of $\langle P_n\rangle_n$ and a measure
	$\nu\leq \mu$ such that $\nu(S)=1$. Let $\langle A_n\rangle_n$ be an enumeration of all co-infinite elements of $\mathbb{A}$. By an easy induction we can construct a selector $S$ of $\langle P_n\rangle_n$ such that $S\setminus A_n$ is infinite for every $n$. But then, using
	Lemma \ref{lem:sikorski-measure} we may extend $\mu$ to $\nu$ such that $\nu(S)=1$. Finally, notice that as a $\sigma$-closed forcing $\mathbb{P}$ does not add new reals and so it does not add new partitions of $\omega$.

	Now, we will prove that $\forces{ \dot{\mu}\text{ is a P-measure}}{\mathbb{P}}$. Consider a decreasing sequence $\langle A_n\rangle_n$ of subsets of $\omega$ and $\mu \in \mathbb{P}$ such that $\{A_n\colon n\in \omega\}$ is contained in the domain of $\mu$. Now, the argument is
	similar to that from Section \ref{sec:transfinite}: by the fact that there are no
	$(\omega,\omega)$-gaps in $\mathcal{P}(\omega)/fin$ we conclude that there is $X\subseteq \omega$ such that $X\subseteq^* A_n$ for each $n$ and $\mu^*(X)=\lim_{n} \mu(A_n)$. By Lemma \ref{lem:sikorski-measure}, there is $\mu' \leq \mu$ such
	that $\mu'(X)=\lim_n \mu(A_n)$.
\end{proof}

\begin{thm}\label{s-m-n-u} Consistently, there is a selective measure but there are no selective ultrafilters.
\end{thm}
\begin{proof}
	We will use the classical random model. Let $\mathcal{U}$ be a selective ultrafilter in the ground model and let $\dot{\mu}$ be the Solovay measure induced by $\mathcal{U}$ (see Section \ref{sec:solovay}). We will work in $V[G]$. By \cite[Theorem
	4.15]{PbnDamian} $\dot{\mu}$ is a
P-measure. We have $\mathcal{U} \subseteq \{X\colon \dot{\mu}(X)=1\}$ (in fact, by \cite[Theorem 5.6]{PbnDamian}, we have even equality here). Since the random forcing is $\omega^\omega$-bounding, $\mathcal{U}$ is a Q-filter in $V[G]$ and so, by Proposition \ref{prop:q-filter}, $\dot{\mu}$ is a Q-measure.
	By Proposition \ref{lem:selective_p_q} we have that $\dot{\mu}$ is a selective measure. 

	But there are no selective ultrafilters in the random model (\cite{Kunen}).
\end{proof}

\begin{thm}\label{manyselective} Suppose that there are infinitely many selective ultrafilters which are pairwise RK-incomparable. Then there is a non-atomic selective measure.
\end{thm}
\begin{proof}
	Let $\langle\mathcal{U}_n\rangle_n$ be a family of pairwise RK-incomparable selective ultrafilters. Let $\mu$ be a measure in the weak$^*$ closure of the set of measures of the form
	\[ \frac{1}{n}\big(\delta_{\mathcal{U}_0}+\dots+\delta_{\mathcal{U}_{n-1}}\big). \]
	This is a non-atomic measure on $\omega$.
		To show that it is selective consider a function $f\colon \omega \to \omega$  and let $S_n$ be a selector of $\langle f^{-1}[n]\rangle_n$ appropriate for $\mathcal{U}_n$.
	\medskip

		\textbf{Claim.} There is a family $\langle A_n\rangle_n$ of pairwise disjoint sets such that $f^{-1}[A_n] \in \mathcal{U}_n$ for each $n$.
		\medskip

		Indeed, since $\langle \mathcal{U}_n\rangle_n$ are RK-incomparable,  $\langle f[\mathcal{U}_n]\rangle_n$ are pairwise different. For every $n$, by the fact that $f[\mathcal{U}_n]$ is a P-point (and so it is not in a closure of any countable subset), there is $A_n \in f[\mathcal{U}_n]$ such that $A_n \notin f[\mathcal{U}_k]$ for $k\ne n$. The family $\langle A_n\rangle_n$ satisfies the desired property.
		\medskip

		Let \[ S = \bigcup_n \big(S_n \cap f^{-1}[A_n]\big). \]
		The set $S$ is a selector of $\langle f^{-1}[n]\rangle_n$ and $\mu(S) = 1$.
\end{proof}

Note that the assumption on non-comparability is crucial. Consider a selective ultrafilter $\mathcal{U}$ and for $n\in \omega$, define $\mathcal{U}_n$ as the ultrafilter generated by  $\{n+A\colon A\in \mathcal{U}\}$. Then every measure in the
weak$^*$ closure of the set of measures of the form
\[ \frac{1}{n}\big(\delta_{\mathcal{U}_0}+\dots+\delta_{\mathcal{U}_{n-1}}\big) \]
is shift invariant and so it cannot be selective (see Proposition \ref{qnotdens}).

\begin{prop}\label{2tocontinuum} If there is a non-atomic Q-measure, then there are $2^\mathfrak{c}$ many coherence classes with Q-points. \end{prop}

\begin{proof} If $\mu$ is a non-atomic Q-measure, then $\mathcal{F} = \{F\colon \mu(F)=1\}$ is a Q-filter. Since every closed infinite subset of $\beta\omega$ contains a copy of $\beta\omega$,  there are $2^\mathfrak{c}$ extensions of $\mathcal{F}$
	to ultrafilters. Clearly, all of them are Q-points. Now, notice that if Q-points are nearly coherent, then they are isomorphic (see e.g. \cite[Corollary 2]{Blass73}). But there are only $\mathfrak{c}$ many possible isomorphism classes.  
\end{proof}

As an immediate corollary of Theorem \ref{manyselective} and Proposition \ref{2tocontinuum} we get the following.

\begin{cor} If there are infinitely many selective ultrafilters which are pairwise RK-incompatible, then there are $2^\mathfrak{c}$ many coherence classes with Q-points.
\end{cor}

Also, we get that the existence of a selective ultrafilter does not guarantee the existence of a non-atomic selective (even Q-) measure.

\begin{cor}\label{heike} Consistently, there is a selective ultrafilter, but there are no non-atomic Q-measures.
\end{cor}
\begin{proof}
	In \cite{Heike} Mildenberger presented a model in which  there are exactly three coherence classes of ultrafilters on $\omega$ and there is a selective ultrafilter. By Proposition \ref{2tocontinuum} there is no non-atomic Q-measure in this model.
\end{proof}

 Similar to the other types of measures, we can give a result about the generic existence of selective measures, which is analogous to the well-known result that every filter base of size less than $\mathfrak{c}$ can be extended to a selective ultrafilter if and only if $\mathrm{cov}(\mathcal{M}) =\mathfrak{c}$ (See \cite{canjargeneric} or \cite[Theorem 4.5.6]{Bartoszynski}). 

 \begin{thm}\label{covM=c}
	$\mathrm{cov}(\mathcal{M}) = \mathfrak{c}$ if and only if every measure $\mu$ defined on a Boolean algebra $\mathbb{B}\subseteq \mathcal{P}(\omega)$ such that $|\mathbb{B}| < \mathfrak{c}$, can be extended to a selective measure on $\omega$.
\end{thm}
\begin{proof}
    For the proof, we will need the following characterization of $\mathrm{cov}(\mathcal{M})$, which follows from \cite[Lemma 2.4.2]{Bartoszynski}.
    \begin{thm*}
    $\kappa < \mathrm{cov}(\mathcal{M})$ if and only if for every collection $\{ f_\alpha : \alpha \in \kappa \}$ of infinite partial functions, there is $g \in \omega^\omega$ such that $g \cap f_\alpha$ is infinite for each $\alpha < \kappa$. Moreover, $\mathrm{cov}(\mathcal{M})$ is the smallest cardinality of a family $\mathcal{F} \subseteq \omega^\omega$ such that for all $g \in \omega^\omega$ there is $f \in \mathcal{F}$ such that $f \cap g$ is finite. 
    \end{thm*}
$(\Rightarrow)$. The proof goes mostly the same as in the proof of Theorem \ref{d=c}, the difference is that in the recursion, we ask that $S$ is a selector instead of a pseudo-selector, but the definition of $\mu', \mathbb{B}'$ and $\mathcal{B}$ are the same. With that in mind, we will only have to define how the selector is constructed. For each $n$, enumerate $A^\alpha_n$ as $\{ a^n(i) : i \in N\}$ where $N$ is either a natural number, or $\omega$. For each $b \in \mathcal{B}$ define $f_b$ a partial function such that $f_b(i) = j$ if and only if $j$ is the first one such that $a^i(j) \notin b$ (which can be done for infinitely many $i$). Since $|\mathcal{B}'| < \mathrm{cov}(\mathcal{M})$, there must be a $g \in \omega^\omega$ such that $g \cap f_b$ is infinite for each $b \in \mathcal{B}$. Without loss of generality, we may assume that each $g(i) \leq |A_i^\alpha|$. Let $S = \{ a^n(g(n)) : n \in \omega \}$. Clearly $S$ is a selector of the partition $(A_n^\alpha)_n$. If $b \in \mathcal{B}$, then there are infinitely many $i$ where $g(i) = f_b(i)$, and therefore $S$ avoids infinitely many elements of $b$, so $S \nsubseteq b$. Therefore, continuing the argument just like in Theorem \ref{d=c}, we get that $\mu'^*(S) = 1 - \sum_n\mu'(A^\alpha_n)$, and Lemma \ref{lem:Sikorski_extension} provides a way to extend $\mu'$ and $\mathbb{B}'$ so all the conditions are satisfied. 

	($\Leftarrow$) Let $\kappa < \mathfrak{c}$ and $\{ f_\alpha : \alpha < \kappa\} \subseteq \omega^\omega.$ We will show that there is $g \in \omega^\omega$ such that $g\cap f_\alpha$ is infinite for each $\alpha$. First, find a partition $\omega =
	\bigcup_{n \in \omega} I_n$ such that $|I_n| = n$, let $X_n = \omega^{I_n}$ and let $X = \bigcup_{n \in \omega} X_n$. Clearly $X$ is a countable set. Let $D_n = \bigcup_{i>n} X_i$ and, for each $\alpha \in \kappa$, let \[ E_\alpha = \{ s \in X :
	\text{ if }s \in X_n, \text{ then }s(i) = f_\alpha (i) \text{ for some }i \in I_n \}.\]  Observe that if $F \in [\kappa]^{<\omega}$, then, for each $i > |F|$, one can build $s \in \omega^{I_i}$ such that $s \in \bigcap_{\alpha \in F} E_\alpha$ and therefore the family $\{ D_n, E_\alpha : n \in \omega, \alpha \in \kappa\}$ generates a filter $\mathcal{F}$ on $X$. Consider the Boolean algebra $\mathbb{B}$ generated by $\{ D_n, E_\alpha : n \in \omega, \alpha \in \kappa\}$. Then, the following function is a free measure on $\mathbb{B}$:
    $$
\delta_\mathcal{F}(A) \begin{cases}
    1 & \text{if }A \in \mathcal{F},\\
    0 & \text{otherwise.}
\end{cases}
    $$
Then $\delta_\mathcal{F}$ can be extended to a full selective measure $\mu$. Therefore, there must be a selector $S$ for the partition $\langle X_n\rangle _n$ such that $\mu(S) = 1$. Without loss of generality, we may assume that $S \cap X_n$ has
	exactly one element. Then, $g = \bigcup S$ is a function. Note that, for each $\alpha$, $\mu(E_\alpha) =1$, and therefore $S \cap E_\alpha$ is infinite, and so $g \cap f_\alpha$ is infinite.
\end{proof}

\section{RB-minimal measures versus Q-measures}\label{sec:minimality}

In this section, we will address the following question:

\begin{prob} Is every RB-minimal measure a Q-measure?
\end{prob}

Of course, this is the case for ultrafilters. The main tool in the proof is the following fact (see \cite[Theorem 1]{Blass73}).

\begin{prop}\label{fixed} Suppose that $f\colon \omega \to \omega$ and an ultrafilter $\mathcal{U}$ is such that $f[\mathcal{U}]=\mathcal{U}$. Then the set of fix points of $f$ is in $\mathcal{U}$.
\end{prop}

This fact cannot be directly generalized for measures: Consider $f(n) = n+1$ and a shift invariant measure $\mu$. Then $f[\mu]=\mu$ but the set of fixed points is empty. 

There is another possible way of generalizing Proposition \ref{fixed} to the realm of measures. Again, suppose that $\mu$ is a measure and $f\colon \omega \to \omega$ is a finite-to-one function such that $f[\mu] = \mu$. We may define an equivalence relation on $\omega$:

\begin{equation}\label{equivalence}
 k_1 \sim k_2 \mbox{ if there are }n_1, n_2\mbox{ such that }f^{(n_1)}(k_1) = f^{(n_2)}(k_2). 
\end{equation}

Choose a transversal $S\subseteq \omega$ of equivalence classes of $\sim$. Define
\[ X = \bigcup\left\{f^{(n)}[S]\colon n>0\right\}. \]

The set $X$ is an attractor of $f$, i.e. for each $k \in \omega$ there is $n$ such that $f^{(n)}(k)\in X$. It is also a selector for $\{f^{-1}[n]\colon n\in \omega\}$. It seems natural to conjecture that $\mu(X)=1$. That would be a natural generalization for Proposition	\ref{fixed}. However this is not true:

Consider the measure from Example \ref{fix}, i.e. $\mu$ such that
\[ \mu(A) = \int \big|A\cap [2^n, 2^{n+1})\big|/2^n \ d\mu. \]
Let $f\colon \omega \to \omega$ be such that $f^{-1}[n] = [2^n, 2^{n+1})$. Then $f[\mu] = \mu$ since
\[ \mu(f^{-1}[A]) = \int \big|f^{-1}[A] \cap [2^n, 2^{n+1})\big|/2^n \ d\mu = \int \chi_A \ d\mu = \mu(A). \]
But the equivalence relation (\ref{equivalence}) has only one equivalence class and each attractor $X$ defined as above is finite. So $\mu(X)=0$.

The above approach reveals an interesting phenomenon: measures which are RB-minimal and which are not Q$^+$-measures are RB-equivalent to a shift invariant measure.

\begin{thm} \label{thm:RBMinDichotomy} Suppose $\mu$ is RB-minimal. If $\mu$ is not a Q$^+$-measure, then there is a shift invariant $\nu \leq_{RB} \mu$.
\end{thm}

\begin{proof}
Let $f\colon \omega \to \omega$ be a finite-to-one function witnessing that $\mu$ is not Q$^+$-measure. By minimality, we may assume that $f[\mu] = \mu$. 

	Let $S$ be a selector of the equivalence relation (\ref{equivalence}). If possible we choose an element $s$ with $f^k(s)=s$ for some $k$.
	Now let $\widehat{S}=\{f^{k}(n):n\in S,\,k\in\omega\}$. Notice that $\widehat{S}$ is a selector for $f$ and so $\mu(\widehat{S})=0$. We can arbitrarily change values on $\widehat{S}$ while preserving the properties of $f$.
	
	Let $\widehat{f}$ be equal to $f$ on $\omega\setminus\widehat{S}$. Enumerate $\widehat{S}$ as $\{s_n:n\in\omega\}$ and let $\widehat{f}(s_{n+1})=s_n$ and $\widehat{f}(s_0)=s_0$.
	
	Notice, that $\widehat{f}$ is a finite-to one function with the following properties:
	\begin{itemize}
		\item $\widehat{f}[\mu]=\mu$,
		\item every selector for $\widehat{f}$ has measure 0,
		\item for every $n\in\omega$ there is $k$ such that $\widehat{f}^{(k)}(n)=s_0$. 
	\end{itemize} 
For each $n\in\omega$, let $k(n)$ denote the minimal $k$ as above.

For each $n$ the set $\{m:k(m)=n\}=k^{-1}[n]$ is finite and so $k$ is finite-to-one. We claim that $\nu=k[\mu]$ is shift invariant.

Indeed, notice that $k^{-1}[n+1]=\widehat{f}^{-1}[k^{-1}[n]]$. Thus for every $A\subseteq\omega$ we have
\[
\nu(A+1)=\mu(k^{-1}[A+1])=\mu\big(\widehat{f}^{-1}[k^{-1}[A]]\big)=\mu(k^{-1}[A])=\nu(A).
\]
\end{proof}

Compare the above theorem with the following proposition. We will say that a measure $\mu$ is \emph{nearly shift invariant} if there is $\nu \leq_{RB} \mu$ such that $\nu$ is shift invariant.

\begin{prop} \label{Q-not-shift}
	If $\mu$ is Q$^+$, then it is neither nearly shift invariant nor RB-above a density measure.\label{qnotdens}
\end{prop}
\begin{proof}First note that the Q$^+$ property is RB-hereditary, so it is enough to show that $\mu$ is neither shift invariant nor a density.
	
	Let $I_n=[2^n,2^{n+1})$ and let $S = \{s_0 < s_1 < \dots\}$ be a selector for $\langle I_n\rangle_n$ of positive measure. Notice that $|S\cap k| < \log_2(k)$ and so $S$ has density 0.
	
	Assume that $\mu$ is shift invariant. Let $S' = \{s_{2k}\colon k\in \omega\}$ and assume (throwing out the first element of $S$ if necessary) that $\mu(S')>0$. Take $n$ such that $\mu(S')>\frac{1}{n}$. Notice that the sets $S',S'+1,S'+2,\cdots S'+n$ are pairwise almost disjoint and of the same measure. Thus $\mu\big(\bigcup_{k=0}^{n} (S'+k)\big)>1$ which is a contradiction.
\end{proof}

\begin{fact}
	If $\mu$ is a Q-measure and $\nu$ is RB-minimal nearly shift invariant measure, then $\lambda=\alpha\mu +(1-\alpha)\nu$ is RB-minimal for any $\alpha\in[0,1]$.
\end{fact}
\begin{proof}
	Given a finite-to-one function $f$, we have to find a finite-to-one $g$ such that $g\circ f[\mu]=\mu$ and $g\circ f[\nu]=\nu$. We may assume that we already have $f[\nu]=\nu$. We can find $S$ such that $f[\mu](S)=1$ and $\nu(S)=0$ (pick a partition witnessing that $\nu$ is not Q$^+$ and a selector $S$ witnessing that $f[\mu]$ is Q). If $g'$ is such that $g'\circ f[\mu]=\mu$ then let $g=g'\restr_S \cup id_{\omega\setminus S}$. Clearly, $g$ is as desired.
\end{proof}
The above fact inspires a decomposition theorem for RB-minimal measures. Say that a measure $\mu$ is \emph{rigid} if there is no way to decompose $\mu = \alpha\mu_0 + (1-\alpha)\mu_1$, with $\mu_0$ nearly shift invariant and $\alpha>0$.  For a $\mu$-positive set $A\subseteq\omega$ we put $\mu\restr_A(B)=\mu(A\cap B)/\mu(A)$.

\begin{prop}\label{decomposition} Every measure $\mu$ can be decomposed into $\mu = \alpha\mu_0 + (1-\alpha)\mu_1$, where $\mu_0$ is a countable convex combination of nearly shift invariant measures, $\mu_1$ is rigid and $\alpha\in[0,1]$. If $\mu$ is additionally a P-measure,
	then we may assume that $\mu_0$ is nearly shift invariant.
\end{prop}

\begin{proof} 
	The first part is proven by a standard argument: we can subsequently decompose $\mu$ using the lack of rigidity until we reach the 'rigid core' (possibly zero). Clearly, we can do it only countably many
	times. 
	
	For the second part consider a maximal pairwise almost disjoint family $\langle X_n\rangle_n$ of sets such that $\mu(X_n)>0$ and $\mu\restr_{X_n}$ is nearly shift invariant. For each $n$ let $f_n\colon X_n \to \omega$ be the function witnessing near shift invariance of $\mu\restr_{X_n}$. Since $\mu$ is a P-measure, there is $Y$ which is almost disjoint with each $X_n$ and such that $\mu(Y) = 1 - \sum_n \mu(X_n)$. By
	maximality $\mu_1 = \mu\restr_Y$ is rigid. Let $f\colon \omega \to \omega$ be defined by $f(k) = f_n(k)$ if $k\in X_n$ and $f(k)=k$ if $k\notin \bigcup_n X_n$. Then $f$ witnesses that $\mu_0 = \sum_n \mu(X_n)\cdot\mu\restr_{X_n}$ is nearly shift
	invariant.
\end{proof}

\begin{cor}\label{thm:RBMinDecomposition} If $\mu$ is RB-minimal, then $\mu = \alpha\mu_0 + (1-\alpha)\mu_1$, where $\mu_0$ is a countable convex combination of nearly shift invariant measures, $\mu_1$ is a Q-measure and $\alpha\in[0,1]$. If $\mu$ is additionally a P-measure, then we may assume that $\mu_0$ is nearly shift invariant.
\end{cor}
\begin{proof} This is a corollary of Proposition \ref{decomposition} and the proof of Theorem \ref{thm:RBMinDichotomy}.
\end{proof}
\subsection{Construction of a shift invariant RB-minimal measure}

We are going to prove the following theorem.

\begin{thm}\label{thm:RBMinNotQ}
	Under CH there is a shift invariant RB-minimal measure. Consequently, under CH there is an RB-minimal measure that is not a Q$^+$-measure.
\end{thm}

Later, we improve the measure to a P-measure and weaken the assumption to $\mathfrak{p}=\mathfrak{c}$. We will use the notation of Section \ref{sec:Mokobodzki}. In particular we will deal with sequences $\langle c_n \rangle_n$ of functionals on $\ell^\infty$ such that $\langle c_n \rangle_n \preceq \langle e_n
\rangle_n$. Recall that for such sequences each $c_n$ has finite support and $c_n(x)=r$ for any sequence constantly equal to $r$ on the support of $c_n$. For simplicity we will sometimes abbreviate $\langle c_n\rangle_n$ to $c$ and follow a similar rule for other sequences of functionals.  The reader may notice that the following lemmas are saying that some particular sets are dense in the order $\preceq^*$.

\begin{lem}
	\label{almostId}
	For any sequence of functionals $\langle c_n\rangle_n\preceq \langle e_n\rangle_n$ and a finite-to-one function $f$, there is a subsequence $\langle c'_n \rangle_n$ and a finite-to-one function $g$ such that
	\[(\forall {x\in \ell^\infty})\; c'_n(x\circ g\circ f)=x(n).\]
\end{lem}
\begin{proof}
	We will inductively construct $g$ in such a way that $g^{-1}[\{n\}]$ are going to be consecutive intervals. Let $g^{-1}[\{0\}]$ be long enough so that the support of $c_0$ is contained in $(g\circ f)^{-1}[\{0\}]$ and put $c_0'=c_0$. For $n>0$ let
	$k_n=1+\max(g\circ f)^{-1}[\{0,1,\cdots n-1\}]$. Since $c\preceq e$, the support of $c_{k_n}$ is disjoint with $[0,k_n)$ and so also with $(g\circ f)^{-1}[\{0,1,\cdots n-1\}]$. Let $g^{-1}[\{n\}]$ be long enough so that the support of $c_{k_n}$ is contained in $(g\circ f)^{-1}[\{n\}]$ and put $c'_n=c_{k_n}$. 
	
	Since $\supp c_{k_n} \subseteq (g\circ f)^{-1}[\{n\}]$ and $x\circ g \circ f$ is constantly equal to $x(n)$ on $(g\circ f)^{-1}[\{n\}]$, we have that $c'_n(x\circ g \circ f)=x(n)$. 	
\end{proof}

\begin{lem}\label{almostInv}
	For every sequence of functionals $\langle c_n\rangle_n\preceq \langle e_n\rangle_n$ and a finite-to-one function $f$, there is a sequence of functionals $\langle d_n \rangle_n\preceq \langle c_n \rangle_n$  and a finite-to-one function $g$ such
	that
	\[(\forall {x\in \ell^\infty})\;|d_n(x)-d_n(x\circ g\circ f)| <\frac{\|x\|_\infty}{2^n}.\]
\end{lem}

\begin{proof}
	Let $c'$ and $g$ be obtained from Lemma \ref{almostId}. Put $c^0_n=c_n$ for each $n$. Now let $c^{m+1}_n(x)=c^m_n(c'(x))$ i.e. if $c_n^m=\sum_{k=n}^{N}\alpha_ke_k$ then $c_n^{m+1}=\sum_{k=n}^{N}\alpha_kc'_k$. Therefore, we have that $c^{m+1}\preceq c'\preceq e$. In this way, and by Lemma \ref{almostId}, we ensure that $c^{m+1}_n(x\circ g \circ f)=c^m_n(x)$.

	Now let \[d_n=\frac{1}{2^{n+1}}\big(c_n^{0}+c_n^{1}+c_n^{2}+\cdots + c_n^{2^{n+1}-1}\big).\]
	Pick any $x\in \ell^\infty$. Notice that 
	\[d_n(x\circ g \circ f)=\frac{1}{2^{n+1}}\big(c_n^0(x\circ g \circ f)+c_n^0(x)+c_n^1(x)+\cdots +c_n^{2^{n+1}-2}(x)\big)\] 
	and so $|d_n(x)-d_n(x\circ g \circ f)|<(2^{-n-1}+2^{-n-1})\|x\|_\infty$, as $|c_n^m(x)|\leqslant \|x\|_\infty$.
\end{proof}
	
\begin{thm}\label{RBminFunc}
	Assume $CH$. For any sequence of functionals $\langle c_n\rangle_n\preceq \langle e_n\rangle_n$ there is a functional $\varphi$ in the intersection of $\{\overline{conv(c_n:n>k)}:k\in\omega\}$\footnote{\,$\overline{\,\cdot\,}$ denotes a closure in the topology of pointwise convergence}  such that for any finite-to-one $f$ there is a finite-to-one $g$ with
	\[(\forall x\in \ell^\infty)\;\;\varphi(x\circ g\circ f) = \varphi(x).\]
\end{thm}
\begin{proof}
	Let $\langle x_\alpha:\alpha\in\omega_1\rangle,\langle f_\alpha:\alpha\in\omega_1\rangle$ enumerate $\ell^\infty$ and all finite-to-one functions in $\omega^\omega$.
	
	We construct a sequence $\langle c^\alpha\rangle_{\alpha\in\omega_1}$ with 
	\begin{enumerate}
		\item $c^\alpha \preceq c$ for any $\alpha<\omega_1$,
		\item $c^\alpha\preceq^* c^\beta$ for $\alpha>\beta$,
		\item $\lim_n c_n^\alpha(x_\alpha)$ exists,
		\item there is a finite-to-one $g_\alpha$ such that 
		\[(\forall x\in \ell^\infty)\;\;|c_n^\alpha(x\circ g_\alpha \circ f)-c_n^\alpha(x)|<\frac{\|x\|_\infty}{2^n}.\]
	\end{enumerate}
	First, for each $\alpha$, we define a sequence $\langle v^\alpha_n\rangle_n$ in the following way: Let $v_n^0=c_n$ and let $v_n^{\alpha}=c_n^\beta$ if $\alpha = \beta + 1$. If $\alpha$ is limit, pick $\langle\alpha_n\rangle_n$ a sequence of ordinals increasing to $\alpha$. Assuming we have defined $c^\xi$ for each $\xi < \alpha$ we put
	\[v^\alpha=c^{\alpha_n}_n.\]
	
	Notice, that in each case $ v^\alpha\preceq e$ and for any $\beta<\alpha$ we have $v^\alpha\preceq^*c^\beta$. Now, pick a subsequence $k_n$ such that $\lim_n v^\alpha_{k_n}(x_\alpha)$ exists (there is one as $\langle v_n^\alpha(x_\alpha) \rangle_n$ is a bounded sequence). Then, we apply Lemma \ref{almostInv} to $f_\alpha$ and the sequence $\langle v^\alpha_{k_n}\rangle_n$ to obtain $g_\alpha$ and $\langle c_n^\alpha\rangle_n$. This finishes our recursive construction.
		
	Finally, let $\varphi(x)=\lim_n c_n^\alpha(x)$ for the minimal $\alpha<\omega_1$ for which this limit exists. According to Fact \ref{coherence}, for any $\beta$ above such $\alpha$, we have $\varphi(x)=\lim_n c_n^\beta(x)$. Using (3) we know that
	$\varphi$ is defined on each element of $\ell^\infty$. A simple reflection argument shows that $\varphi$ is indeed a functional.
	
	We will show that $\varphi \in \bigcap\{\overline{conv(c_n:n>k)}:k\in\omega\}$: Pick any $x\in \ell^\infty$. There is $\alpha<\omega_1$ s.t. $x=x_\alpha$. Thus $\varphi(x)=\lim_n c_n^\alpha(x)$ and so for any $k\in\omega$, we have that $\varphi(x)\in \overline{conv(c_n^\alpha(x):n>k)}$. Hence $\varphi\in \overline{conv(c_n^\alpha: n>k)}\subseteq \overline{conv(c_n: n>k)}$.
	
	To see the last part, let $f$ be a finite-to-one function. There is $\alpha$ such that $f_\alpha=f$. Pick any $k\in \omega$. For every functional $\psi \in conv(c^\alpha_n:n>k)$  we have 
	\[|\psi(x\circ g_\alpha\circ f_\alpha)-\psi(x)|\leqslant\frac{\|x\|_\infty}{2^k}.\]
	So the same is true for $\varphi$ since $\varphi \in \overline{conv(c^\alpha_n:n>k)}$.	
\end{proof}

As a corollary we get Theorem \ref{thm:RBMinNotQ} promised at the beginning of this section:

\begin{cor}\label{RBminMeas}
	 (CH) There is an RB-minimal shift invariant measure extending the density.
\end{cor}
\begin{rem}
	Let $\mu$ be a RB-minimal shift invariant measure. Let $f$ be finite-to-one. If the lengths of intervals contained in $f^{-1}[\{n\}]$ are not bounded, then by argument as in Proposition \ref{qnotdens} we have that there is no
	$\mu$-positive selector for $\langle f^{-1}[\{n\}] \rangle$. In particular, if $g\circ f[\mu]=\mu$, then $\mu( \mathrm{Fix}(g\circ f))=0$. Thus, there is a $\mu$-positive set $A$ with $g\circ f[A]\cap A=\emptyset$. However $g\circ f$ in some sense preserves the structure of $\mu$ on subsets of $A$. 
	
	What this shows is that while Q-measures obtain RB-minimality by ignoring a lot of the domain (having 'rare' sets of measure 1), this $\mu$ achieves the same by having a lot of nontrivial inner 'symmetries'. 
\end{rem}

Now, we will show that a simple modification of the argument above yields a similar result for the Rudin-Keisler order; the main idea is to add a restriction to additionally ask that the resulting measure is a P-measure.

In the following, if $c$ is a functional in $\ell^\infty$ and $A\subseteq \omega$, then $c(A)$ means $c(\chi_A)$.

\begin{lem} Suppose $\langle c_n\rangle_n \preceq \langle e_n\rangle_n$ is a sequence with pairwise disjoint supports and let $\langle A_k\rangle_k$ be a partition of $\omega$ such that $r_k = \lim_n c_n(A_k)$ exists for every $k$. Then, there is $P\subseteq \omega$ a pseudo-union of $\langle A_k\rangle_k$ such that $\lim_i c_n(P) = \sum_k r_k$.
\end{lem}

\begin{proof}
	Let $\langle n_k\rangle_k$ be an increasing sequence such that, for any $m\geqslant n_k$, we have
	\[\big|c_m(A_0\cup\dots\cup A_k)-(r_0+\dots+r_k)\big|<2^{-k}\]
	and define $C_k=\bigcup_{m=n_k}^{n_{k+1}-1}\supp c_m$.
	Finally let \[ P = \bigcup_k\Big(C_k\cap (A_0 \cup \dots \cup A_k) \Big) \cup \Big(\omega\setminus\bigcup_n \mathrm{supp}(c_n)\Big).  \]
	Clearly $P$ is a pseudo-union of $\langle A_k\rangle_k$ and for $n_k\leqslant m<n_{k+1}$ we have $c_m(P)=c_m(A_0\cup\dots\cup A_k)$.
	
	Fix $\varepsilon>0$ and take $M$ such that $2^{-M} + \sum_{m>M} < \varepsilon$. It is now easy to see that for any $m>n_M$ we have 
	\[\Big|c_m(P)-\sum_kr_k\Big| < \varepsilon.\]
\end{proof}

\begin{thm}\label{RKminMeas} (CH) There is a shift invariant RK-minimal measure.
\end{thm}

\begin{proof} We proceed similarly as in the proof of Theorem \ref{RBminFunc} but additionally, at each step $\alpha$, we make sure that $\langle c^\alpha_n\rangle_n$ has disjoint supports.

In this way we finally obtain a (shift invariant, RB-minimal) P-measure: for each infinite partition $\langle A_k\rangle_k$ of $\omega$ we will find $\alpha$ such that $\lim_n c_n^\alpha(A_k)$ exists for all $k$. Then the above lemma gives us a pseudo-union of the appropriate measure. 

But, by Fact \ref{fact:collbelowp} a measure $\nu$ which is RB-minimal and P is in fact RK-minimal.
\end{proof}

We will now explore the possibility of whether CH can be weakened. 

\begin{rem}
	Under Filter Dichotomy, all measures are nearly ultra and there are no Q-points, hence there are no RB-minimal measures.
\end{rem}

Let us define $\preceq^*_{\mathbb{Q}}$ ($\preceq_\mathbb{Q}$) by restricting $\preceq^*$ ($\preceq$, respectively) to rational convex combinations. In what follows $\mathfrak{p}$ is the \emph{pseudo-intersection number} (see \cite{Blass}). 
\begin{prop}\label{pclosed}
	For any $\gamma<\mathfrak{p}$, for each $\preceq^*$-descending sequence $\langle c^\alpha: \alpha<\gamma\rangle$ with $c^\alpha\preceq_{\mathbb{Q}} e$, there is $c\preceq_{\mathbb{Q}} e$ with $c\preceq^*c^\alpha$ for each $\alpha<\gamma$.
\end{prop}
Before we prove this fact, let us see that it is enough to weaken the assumption of $CH$ in Theorems \ref{RBminFunc},\ref{RBminMeas} and \ref{RKminMeas} to $\mathfrak{p}=\mathfrak{c}$. Notice that the only obstruction stopping us from continuing the
construction beyond $\omega_1$ and up to $\mathfrak{c}$ is the difficulty of obtaining $v^\alpha \preceq^* c^\beta $ given $\langle c^\beta:\beta<\alpha\rangle$. This precisely is covered by the above proposition assuming that
$\mathfrak{p}=\mathfrak{c}$. Notice also that beyond the starting sequence only Lemma \ref{almostInv} results in non-trivial convex combinations and they both use only rationals. 
\begin{proof}
	Given a $\preceq^*$-descending sequence $\langle c^\alpha: \alpha<\gamma\rangle$ we define a poset $\mathbb{M}$ by
	\[\mathbb{M} = \{\langle s,F\rangle: s\text{ is a finite sequence extendible to } \sigma\preceq_{\mathbb{Q}} e\text{ and } F\subseteq \gamma \text{ is finite}\} \]
	and put $\langle s',F'\rangle \leqslant \langle s,F\rangle$ if and only if $s'$ extends $s$, $F\subseteq F'$ and
	\[\forall n\in\dom(s')\setminus\dom(s)\,\forall\alpha\in F\; s'_n\in conv(c_k^\alpha:k\geqslant n).\]
	Furthermore, for each $\alpha$ put $D_\alpha=\{\langle s,F\rangle\in\mathbb{M}: \alpha\in F\}$ and $D^n=\{\langle s,F\rangle\in\mathbb{M}: n\in\dom(s)\}.$ It is easy to see that each $D_\alpha$ and $D^n$ is dense in $\mathbb{M}$. 
	
	Notice also that any finite set of elements of $\mathbb{M}$ sharing the same $s$ has a common lower bound (simply take the union of $F$'s), thus $\mathbb{M}$ is $\sigma$-centered. 
	
	By \cite{Bell1981} (see also \cite[Theorem 7.12]{Blass}) we know that $MA_{<\mathfrak{p}}(\sigma\text{-centered})$ is true, which gives us a filter $G$ on $\mathbb{M}$, generic for the family of all $D_\alpha$'s and $D^n$'s.
	
	Let $c=\bigcup\{s:\langle s,F\rangle \in G\}$. Then $c\preceq_{\mathbb{Q}} e$ and given $\alpha<\gamma$ we can find $\langle s,F\rangle \in G\cap D_\alpha$ which means that for any $n\not\in\dom(s)$ we have $c_n\in conv(c^\alpha_k:k\geqslant n)$.
\end{proof}

Note that the construction of a universally measurable measure can be done under the assumption $\mathrm{cov}(\mathcal{M})=\mathfrak{c}$ (see
\cite[538S]{Fremlin5}) and in the random model (see \cite[553N]{Fremlin5}), so one can aim to improve the above result.

\subsection{Measurability}
We finish this section with several remarks concerning the measurability (as functions from $2^\omega$ to $\mathbb{R}$) of the measures.

\begin{rem}\label{universallymeasurable} It is possible to modify the construction of Theorem \ref{RBminFunc} so that we may additionally assume that the functional obtained is universally measurable. Indeed, begin with enumerating all bounded Radon measures on
	$2^\omega$ as $\langle\mu_\alpha:\alpha<\omega_1\rangle$. Then given $\langle v_n^\alpha\rangle_n$ we pick a subsequence ensuring that $\lim_n v_{k_n}^\alpha(x)$ exists $\mu_\alpha$-almost everywhere. This can be done by the reflexivity of the
	space $L^2(\ell^\infty,\mu_\alpha)$ in which $\langle v_n^\alpha\rangle_n$ is a bounded sequence. Then, using Mazur's lemma, we may obtain that this sequence converges in norm. More details on this approach can be found in \cite{Godefroy}.
\end{rem}

Measurability is yet another property distinguishing Q-measures from the measure of Theorem \ref{thm:RBMinNotQ}. We enclose a sketch of the proof of the following proposition although it follows from a much more general result of Mokobodzki
who proved that rapid filters are not Lebesgue measurable (see \cite{Barto}).

\begin{prop}\label{Q-not-meas} If $\mathcal{F}$ is a Q-filter, then it is not Lebesgue measurable. Consequently, Q-measures are not Lebesgue measurable. 
\end{prop}

\begin{proof} Since the only possibility for a free filter to be measurable is to have Lebesgue measure 0, it is enough to show that $\mathcal{F}$ intersects every closed subset of $2^\omega$ of positive Lebesgue measure. In other words, we have to show that if $U$ is an open set which is not of full measure, then there is $X\in \mathcal{F}
	\setminus U$. Let $U = \bigcup_n C_n$, where each $C_n$ is a basic clopen.

	Without loss of generality we may assume that for each $n$ there is a finite $F_n$ such that $C_n = \{x\in 2^\omega\colon \forall i\in F_n \ x(i) = 1\}$. Also, we may assume that $\langle F_n\rangle_n$ are pairwise disjoint and of cardinality at
	least $2$. Applying Q-property, we get that there is $X\in \mathcal{F}$ such that $|X\cap F_n|\leq 1$ for every $n$. Clearly, $X\notin U$. 
\end{proof}

\begin{cor} No universally measurable measure has a Q-measure RK-below. 
\end{cor}
\begin{proof} We just have to see that if $\mu$ is universally measurable and $\nu\leq_{RK} \mu$, then $\nu$ is universally measurable, too. Then the corollary follows from Proposition \ref{Q-not-meas}. 
	
	So, let $\mu$ be universally measurable, and let $\nu = f[\mu]$ for some function $f\colon \omega \to
	\omega$. 
	Notice that the function $h\colon 2^\omega \to 2^\omega$ defined by 
	$ h(\chi_A) = \chi_{f^{-1}[A]}$  
	is continuous and so it is universally measurable. So, the measure $\nu = \mu \circ h$ is a composition of universally measurable functions, and so it is universally measurable (see e.g. \cite[434Df]{Fremlin4}).
\end{proof}

One might ask how poorly Q-filters behave in terms of measurability, for instance, whether they are measurable with respect to other natural measures on the Cantor set. In light of Proposition \ref{Q-not-meas}, it may be surprising to note that
selective ultrafilters appear to be more easily measurable under certain measures than other types of ultrafilters:

\begin{rem} Bartoszy\'{n}ski, in \cite{Barto}, considered measures on $2^\omega$ of the form $\mu_p = \prod_n p_n \delta_0 + (1-p_n) \delta_1$ (notice that if $p$ is constantly $1/2$, then $\mu_p$ is the standard Haar measure on $2^\omega$). He
	showed, among other things, that if $\sum_n p_n < \infty$, then all filters are $\mu_p$-measurable. It is not clear to us if there is a sequence $\langle p_n\rangle_n$ converging to $0$ and an ultrafilter $\mathcal{U}$ such that $\mathcal{U}$ is
	not $\mu_p$-measurable. However, selective ultrafilters are always measurable with respect of such measures as, by a result of Mathias (see \cite{Mathias}), these ultrafilters intersect every tall analytic ideal\footnote{In fact, to proceed it is enough if
	an ultrafilter intersects every tall summable ideal, that is if it is a semi-selective ultrafilter.}. So, if $\mathcal{U}$ is selective, then there is $X\in \mathcal{U}$ such that $\sum_{n\in X} p_n < \infty$. By \cite[Theorem 1.3]{Barto} it means that $\mathcal{U}$ must be $\mu_p$-measurable. 
\end{rem}

\section[Between Q and Q+]{Between Q and \texorpdfstring{$Q^+$}{Q+}}\label{sec:Q+}

In this section we will present an attempt to unify natural variations on the definition of a Q-measure. Below we contribute to the discussion about the connections between those variations, some of which were already considered in \cite{Luis}. All partitions mentioned in this section are partitions of $\omega$ into finite sets.

It is not difficult to see that the  Q$^+$-measures form a convex, dense (if non-empty) subset of the set of measures on $\omega$:

\begin{fact}\label{convex-comb}
	If $\mu$ is a Q$^+$-measure, then for every measure $\nu$ and any $\alpha\in(0,1]$, $\lambda=\alpha \mu + (1-\alpha)\nu$ is also Q$^+$.
	Moreover, if $\nu$ is not Q$^+$ then there is a partition that has no selector with $\lambda$-measure larger than $\alpha$.
\end{fact}

The following observation is pivotal to the characterization of Q$^+$-measures.
\begin{thm}
	Given a measure $\mu$ such that for any partition there is a selector of measure strictly greater than $\delta$, there is $\varepsilon>\delta$ such that any partition has a selector of measure no less than $\varepsilon$.
\end{thm}
\begin{proof}
	Assume no such $\varepsilon$ exists. Take $\langle A_n^m\rangle_n$ to be a partition such that every selector $S$ of $\langle A_n^m\rangle_n$ satisfies $\mu(S)<\frac{1}{m}+\delta$. 
	
	Define an interval partition $\langle B_n\rangle_n$ in the following way. Put $B_0=[0,\max A_0^0]$. Having defined $B_n$ let 
	\[b_n=\max\bigcup\Big\{A_k^m: m\leqslant n+1, A^m_k\cap \bigcup_{l\leqslant n} B_l\neq\emptyset\Big\}+1\]
	and define $B_{n+1}=(\max B_n,b_n].$
	Now there is an increasing sequence $k_n$ such that 
	\[\mu\left(\bigcup_n B_{k_n}\right)=0.\]
	
	Let $C_{2n+1}=B_{k_n}$ and let $C_{2n}$ be intervals filling the gaps.
	Take $S$ to be any selector of $\langle C_n\rangle_n$. Notice that for each $m$ for all but finitely many $n$ we have $|S\cap A_n^m|\leqslant 2$. Define $S'$ as $S\cap \bigcup_n C_{2n}$. Then $S'$ is a selector (perhaps up to a finite set) for $\langle A_n^m\rangle_n$ for every
	$m$. Thus $\mu(S')<\frac{1}{m}+\delta$ for all $m$. Therefore, $\mu(S)=\mu(S')\leq \delta$, a contradiction. 
\end{proof}

For every measure we can find $\varepsilon\in[0,1]$ with the following properties.
\begin{enumerate}
	\item For every $\delta<\varepsilon$, any partition has a selector of measure at least $\delta$.
	\item There is a partition, all selectors of which have measure at most $\varepsilon$.
\end{enumerate} 
We will call such $\varepsilon$ the \emph{Q-value} of the measure (denoted by $Q(\mu)$) and a partition from point (2) a \emph{critical partition}. Notice that a measure $\mu$ is Q$^+$ if and only if $Q(\mu)>0$. 

We will say that a measure is an \emph{almost} Q-measure if it has Q-value of 1 and we will call measures \emph{exact} Q$^+$-measures if any partition has a selector of measure at least equal to the Q-value. Notice that Q-measures are those almost Q-measures which are exact.

We now show that almost Q-measures exist if and only if Q$^+$-measures do. Recall that $\mu\restr_A(B)=\mu(A\cap B)/\mu(A)$ for a $\mu$-positive set $A$.
\begin{prop}\label{thm:Q+almostQ}
	Given a Q$^+$-measure $\mu$ and a positive selector $S$ of a critical partition of $\mu$, we have that $\mu\restr_S$ is almost Q. Additionally, if $\mu$ is exact or a P-measure, then $\mu\restr_S$ is a Q-measure. 
\end{prop}
\begin{proof}
	Let $\varepsilon = Q(\mu)$. Take $\langle A_n\rangle_n$, a critical partition of $\mu$, and let $S$ be a selector for $\langle A_n\rangle_n$ with $\mu(S)>0$. 
	
	Let $\langle B_n\rangle_n$ be any partition of $\omega$ into finite sets and fix $c<1$ (or $c=1$ if $\mu$ is exact).
	Let \[C_n=\bigcup\{A_k:A_k\cap B_n\cap S\neq \emptyset\}\]
	and let $S'$ be any selector of $\langle C_n\rangle_n$ with $\mu(S')\geqslant \varepsilon - (1-c)\mu(S)$. Then $S'$ is also a selector for $\langle A_n\rangle_n$. Put $A=\bigcup\{A_k:A_k\cap S'\neq\emptyset\}$. 
	
	Since $S'\cup (S\setminus A)$ is a selector for $\langle A_n\rangle_n$ we know that $\mu(S'\cup (S\setminus A))\leqslant\varepsilon$ and so $\mu(S\setminus A)\leqslant (1-c)\mu(S)$. Notice that $S\cap A$ is a selector for $\langle C_n\rangle_n$ and also for $\langle B_n\rangle_n$, but $\mu\restr_S(A\cap S)=1-\mu(S\setminus A)/\mu(S)\geqslant c$.
	
	If $\mu$ is a P-measure then so is $\mu\restr_S$ and the conclusion follows from Proposition \ref{thm:almQ+P=Q} below.
\end{proof}
\begin{cor}
	Every Q$^+$-measure is of the form $\delta\nu + \sum_n \alpha_n\mu_n$ where $\nu$ is not Q$^+$ and each $\mu_n$ is almost Q.
\end{cor}
\begin{proof}
	Recursively pick disjoint positive selectors $\langle S_\alpha\rangle _\alpha$ for a fixed critical partition and define $\mu_\alpha = \mu\restr_{S_\alpha}$. By  Proposition \ref{thm:Q+almostQ} each $\mu_\alpha$ is almost Q. This procedure has to terminate after countably many steps. Then let $\nu = \mu\restr_{\omega\setminus \bigcup_\alpha S_\alpha}$.
\end{proof}
Furthermore, using Fact \ref{convex-comb} for an almost Q-measure, we can observe the following.
\begin{cor}
	If there exists a $Q^+$-measure then any Q-value in the range $[0,1]$ is attained by some measure.
\end{cor}

One could ask whether the almost Q-measures are actually different from Q-measures. To this end, we provide a partial negative result. 
\begin{prop}\label{thm:almQ+P=Q}
	Every P-measure which is almost Q, is a Q-measure.
\end{prop}
\begin{proof}
	Fix a partition $\langle A_n\rangle_n$  of $\omega$ into finite sets. For each $n$ fix a selector $S_n$ of $\langle A_n\rangle_n$ with $\mu(S_n)>1-2^{-n-2}$. We want to first make $\langle S_n \rangle_n$ almost increasing.
	
	Notice that for each $n$
	\[\sum_{k\geqslant n}\mu(\omega\setminus S_k) < \frac{1}{2^{n+1}}.\]
	Let $Z_n$ be a pseudo-union of $\langle \omega\setminus S_k:k>n\rangle$ with $\mu(Z_n)<2^{-n-1}$. For $n>0$ we can choose $Z_n \subseteq Z_{n-1}$. 
	
	Now let $S'_n=S_n\setminus Z_n$. Then $\mu(S'_n)>1-2^{-n}$ and $S'_n \subseteq^* S'_{n+1}$.
	
	Let $B_{-1}=\emptyset$ and
	\[B_n=\bigcup\{A_k: A_k\cap (S'_n\setminus S'_{n+1}) \neq \emptyset \text{ or }\min A_k<n\}\cup B_{n-1}.\]
	
	Notice that $S_n'\setminus B_{k-1}\subseteq S_k'$ for $k\geqslant n$ and $\bigcup_n B_n=\omega$. Now define 
	\[S =\bigcup_{n}(S'_n\setminus B_{n-1}) = \bigcup_n (S'_{n}\cap (B_{n}\setminus B_{n-1})).\]
	 Then for each $n$ we have $S\supseteq^* S'_n$ and so $\mu(S)=1$. Furthermore for any $k$ there is a unique $n$ such that $A_k\subseteq B_{n}\setminus B_{n-1}$ and so $|S\cap A_k|=|S'_n\cap A_k|\leqslant 1$.	
\end{proof}

It is not clear to us how to modify the proof above to show that all the P-measures which are Q$^+$ are exact Q$^+$-measures. 

To relate these facts to the terminology and results of \cite{Luis}, recall that a measure is a \emph{$\delta$-strong  Q-measure} if any partition has a selector of measure at least $\delta$ and a
\emph{fit Q$^+$-measure} if for every partition and $\delta<\frac{1}{2}$ there is a finite union of selectors of measure greater than $\delta$. Notice that any $\mu$ is a $\delta$-strong Q-measure for any $\delta<Q(\mu)$ and exact measures are those $\mu$ that are $Q(\mu)$-strong Q-measures.  We have thus shown the
following.

\begin{cor}\,
	\begin{itemize}
		\item If there
			is a Q$^+$-measure, then any two 'strengths' can be separated (i.e. for every $0< \varepsilon < \delta \leq 1$ there is a measure which is $\varepsilon$-strong but not $\delta$-strong),
		\item every Q$^+$-measure is
$\varepsilon$-strong for some $\varepsilon>0$,
		\item  fit Q-measures exists if and only if Q$^+$-measures exist.
	\end{itemize}
\end{cor}

\section{Questions}

It seems that the research on orderings of measures on $\omega$ may be continued in many different directions. Below we enlist some of the problems.

\begin{prob} Is it consistent that there is a Q$^+$-measure (so also an almost Q-measure) while there are no Q-measures? 
\end{prob}

Notice that the negative answer to the above would yield the negative answer to \cite[Question 1]{Luis}: is it consistent that there is a Q$^+$-measure but there are no Q-points. The above problem is interesting also because, as we have shown in
Section \ref{sec:Q+}, the existence of all the weaker forms of Q-measures (Q$^+$, almost Q-measures, fit Q-measures) is equivalent.

In a similar fashion, we can ask about the relation between Q-measures and  RB-minimal that are not Q$^+$.

\begin{prob}
	Is it consistent that there is an RB-minimal measure while there are no Q-measures (or even Q$^+$)? Conversely is it consistent that RB-minimal measures do exists and all of them are Q?
\end{prob}

In this article we were interested mostly in the minimal elements of the considered orderings. However, as in the case of the ultrafilters, one can ask more general questions about the structure of those orderings (in the case of ultrafilters usually
the P-points were studied in this context).

\begin{prob} How different is the structure of the Rudin-Blass (or Rudin-Keisler) ordering of non-atomic (P-)measures and the structure of Rudin-Blass (Rudin-Keisler) ordering of (P-)ultrafilters? Can we embed
	$\mathcal{P}(\omega)/fin$ into the set of non-atomic measures with the Rudin-Keisler ordering?
\end{prob}

Theorem \ref{manyselective} suggests another natural question:

\begin{prob}\label{RBQ} Suppose that there are infinitely many RB-incomparable Q-points. Is there a non-atomic Q-measure?
\end{prob}

If the answer is positive, then we would know that whenever there are infinitely many RB-incomparable Q-points, then there are $2^\mathfrak{c}$ of them, see Proposition \ref{2tocontinuum}. As we know (see \cite{Heike} and the proof of Proposition
\ref{heike}) there is a model with exactly three coherence classes of Q-points. So, Problem \ref{RBQ} is connected to the following question, which seems to be interesting on its own. 

\begin{prob}
Is there a model with exactly $\omega$ many coherence classes of Q-points?
\end{prob}

Now, let us switch to the terminology of Section \ref{sec:Mokobodzki}. We may define the poset
\[ \mathbb{P} = (\{ c\colon c \preceq e \}, \preceq^*). \]
Every infinite subset $A$ of $\omega$ may be identified with a sequence $c^A$ in $\ell_\infty^*$: just let $c^A_n = e_{k_n}$, where $(k_n)$ is an increasing enumeration of elements
of $A$. Then, $c^A \preceq^* c^B$ if and only if $A \subseteq^* B$ and so $\mathcal{P}(\omega)/fin$ naturally embeds in $\mathbb{P}$. In fact we can call $\mathbb{P}$ a 'convex version' of $\mathcal{P}(\omega)/fin$ (although mind that $\mathbb{P}$ does not
have the Boolean structure). 
From this perspective, the proof of Theorem \ref{thm:RBMinNotQ} comes down to showing that the above forcing generically adds a Rudin-Keisler minimal measure which is shift invariant and in Proposition \ref{pclosed} we prove that $\mathbb{P}$ is at least $\mathfrak{p}$-closed (and the
forcing $\mathbb{M}$ from the proof is a kind of a 'convex' version of the classical Mathias forcing). Also, the filters in $\mathbb{P}$ induce naturally partial functionals in $\ell_\infty^*$.

This remark suggests a lot of questions about the structure of $\mathbb{P}$, for example about the 'convex' versions of some classical cardinal coefficients.

\begin{prob} What is the distributivity of $\mathbb{P}$ (convex $\mathfrak{h}$)? What is the minimal size of a maximal antichain in $\mathbb{P}$
	(convex $\mathfrak{a}$)?
\end{prob}

Notice that since a forcing notion collapses $\mathfrak{c}$ to its distributivity, the answer to the last question could give us more information what do we need to assume to get an existence of a RK-minimal measure which is shift invariant.

\section{Acknowledgements}

We would like to thank Luis S\'aenz, Szymon Smolarek and Andr\'es Uribe-Zapata for fruitful discussions concerning measures on $\omega$.

\bibliographystyle{alpha}
\bibliography{bib-ppoint}

@book{Blass-thesis,
	AUTHOR = {Blass, Andreas},
	TITLE = {Ordering of ultrafilters},
	NOTE = {Ph.D. thesis},
	PUBLISHER = {Harvard University},
	YEAR = {1970},
}

@book{Fremlin4,
    AUTHOR = {Fremlin, D. H.},
     TITLE = {Measure theory. {V}ol. 4},
      NOTE = {Topological measure spaces. Part I, II,
              Corrected second printing of the 2003 original},
 PUBLISHER = {Torres Fremlin, Colchester},
      YEAR = {2006},
     PAGES = {Part I: 528 pp.; Part II: 439+19 pp. (errata)},
      ISBN = {0-9538129-4-4},
   MRCLASS = {28-02 (46G10)},
  MRNUMBER = {2462372},
MRREVIEWER = {Heinz\ J. E. K\"onig},
}

@book{Fremlin5,
    AUTHOR = {Fremlin, D. H.},
     TITLE = {Measure theory. {V}ol. 5. {S}et-theoretic measure theory.
              {P}art {I}},
      NOTE = {Corrected reprint of the 2008 original},
 PUBLISHER = {Torres Fremlin, Colchester},
      YEAR = {2015},
     PAGES = {329},
      ISBN = {978-0-9538129-5-0},
   MRCLASS = {03-02 (03-01 03E15 03E70 28-01 28-02 28A05 28C15)},
  MRNUMBER = {3723040},
MRREVIEWER = {Klaas\ Pieter\ Hart},
}

@book{Schechter,
    AUTHOR = {Schechter, Eric},
     TITLE = {Handbook of analysis and its foundations},
 PUBLISHER = {Academic Press, Inc., San Diego, CA},
      YEAR = {1999},
     PAGES = {1 CD-ROM Version 1 (Windows, Macintosh and UNIX)},
      ISBN = {0-12-622765-9},
   MRCLASS = {00A20 (00A05 03-01 46-01 54-01)},
  MRNUMBER = {1731414},
}

@article{Luis,
	AUTHOR = {Avil\'es, Antonio and Martinez-Cervantes, Gonzalo and Poveda, Alejandro and S\'aenz, Lu\'is},
	TITLE = {A {B}anach space with {L}-orthogonal sequences but without {L}-orthogonal elements},
	YEAR = {2025},
	JOURNAL = {preprint},
}

@article{Verner-Dilip,
    AUTHOR = {Raghavan, Dilip and Verner, Jonathan L.},
     TITLE = {Chains of {P}-points},
   JOURNAL = {Canad. Math. Bull.},
  FJOURNAL = {Canadian Mathematical Bulletin. Bulletin Canadien de
              Math\'ematiques},
    VOLUME = {62},
      YEAR = {2019},
    NUMBER = {4},
     PAGES = {856--868},
      ISSN = {0008-4395,1496-4287},
   MRCLASS = {03E50 (03E05 54D80)},
  MRNUMBER = {4028492},
MRREVIEWER = {Klaas\ Pieter\ Hart},
       DOI = {10.4153/s0008439519000043},
       URL = {https://doi.org/10.4153/s0008439519000043},
}

@article{Borisa-Dilip,
    AUTHOR = {Kuzeljevic, Borisa and Raghavan, Dilip},
     TITLE = {A long chain of {P}-points},
   JOURNAL = {J. Math. Log.},
  FJOURNAL = {Journal of Mathematical Logic},
    VOLUME = {18},
      YEAR = {2018},
    NUMBER = {1},
     PAGES = {1850004, 38},
      ISSN = {0219-0613,1793-6691},
   MRCLASS = {03E50 (03E05 03E20 03E35 03E40 54D35 54D80)},
  MRNUMBER = {3809586},
MRREVIEWER = {Andrzej\ Ros\l anowski},
       DOI = {10.1142/S0219061318500046},
       URL = {https://doi.org/10.1142/S0219061318500046},
}

@article{Laflamme-RB,
    AUTHOR = {Laflamme, Claude and Zhu, Jian-Ping},
     TITLE = {The {R}udin-{B}lass ordering of ultrafilters},
   JOURNAL = {J. Symbolic Logic},
  FJOURNAL = {The Journal of Symbolic Logic},
    VOLUME = {63},
      YEAR = {1998},
    NUMBER = {2},
     PAGES = {584--592},
      ISSN = {0022-4812,1943-5886},
   MRCLASS = {04A20 (03E05 54D40 54F15)},
  MRNUMBER = {1627310},
MRREVIEWER = {\v Zarko\ Mijajlovi\'c},
       DOI = {10.2307/2586852},
       URL = {https://doi.org/10.2307/2586852},
}

@book{TheoryOfCharges,
    AUTHOR = {Bhaskara Rao, K. P. S. and Bhaskara Rao, M.},
     TITLE = {Theory of charges},
    SERIES = {Pure and Applied Mathematics},
    VOLUME = {109},
      NOTE = {A study of finitely additive measures,
              With a foreword by D. M. Stone},
 PUBLISHER = {Academic Press, Inc. [Harcourt Brace Jovanovich, Publishers],
              New York},
      YEAR = {1983},
     PAGES = {x+315},
      ISBN = {0-12-095780-9},
   MRCLASS = {28A12 (46E27 46G10)},
  MRNUMBER = {751777},
}

@book{Christensen,
    AUTHOR = {Christensen, J. P. R.},
     TITLE = {Topology and {B}orel structure},
    SERIES = {North-Holland Mathematics Studies},
    VOLUME = {Vol. 10},
      NOTE = {Descriptive topology and set theory with applications to
              functional analysis and measure theory,
              Notas de Matem\'atica, No. 51. [Mathematical Notes]},
 PUBLISHER = {North-Holland Publishing Co., Amsterdam-London; American
              Elsevier Publishing Co., Inc., New York},
      YEAR = {1974},
     PAGES = {iii+133},
   MRCLASS = {54H05 (28A05)},
  MRNUMBER = {348724},
MRREVIEWER = {J.\ C.\ Oxtoby},
}

@article{Godefroy,
    AUTHOR = {Godefroy, Gilles},
     TITLE = {Convex combinations of measurable functions and axioms of set
              theory},
   JOURNAL = {Linear Nonlinear Anal.},
  FJOURNAL = {Linear and Nonlinear Analysis. An International Journal},
    VOLUME = {2},
      YEAR = {2016},
    NUMBER = {2},
     PAGES = {175--187},
      ISSN = {2188-8159,2188-8167},
   MRCLASS = {03E35 (03E05 03E15 03E50 28C15 54E52)},
  MRNUMBER = {3638638},
MRREVIEWER = {Piotr\ Borodulin-Nadzieja},
}

@incollection{Meyer,
    AUTHOR = {Meyer, P. A.},
     TITLE = {Limites m\'ediales, d'apr\`es {M}okobodzki},
 BOOKTITLE = {S\'eminaire de {P}robabilit\'es, {VII} ({U}niv. {S}trasbourg,
              ann\'ee universitaire 1971--1972)},
    SERIES = {Lecture Notes in Math.},
    VOLUME = {Vol. 321},
     PAGES = {198--204},
 PUBLISHER = {Springer, Berlin-New York},
      YEAR = {1973},
   MRCLASS = {28A20},
  MRNUMBER = {404564},
MRREVIEWER = {J.\ C.\ Taylor},
}

@article{Douwen,
    AUTHOR = {van Douwen, Eric K.},
     TITLE = {Finitely additive measures on {${\bf N}$}},
   JOURNAL = {Topology Appl.},
  FJOURNAL = {Topology and its Applications},
    VOLUME = {47},
      YEAR = {1992},
    NUMBER = {3},
     PAGES = {223--268},
      ISSN = {0166-8641,1879-3207},
   MRCLASS = {28A10 (28A60 43A07 46A22)},
  MRNUMBER = {1192311},
MRREVIEWER = {K.\ P. S. Bhaskara Rao},
       DOI = {10.1016/0166-8641(92)90032-U},
       URL = {https://doi.org/10.1016/0166-8641(92)90032-U},
}

@book{Choquet,
    AUTHOR = {Choquet, Gustave},
     TITLE = {Lectures on analysis. {V}ol. {I}: {I}ntegration and
              topological vector spaces},   
 PUBLISHER = {W. A. Benjamin, Inc., New York-Amsterdam},
      YEAR = {1969},
     PAGES = {Vol. I: xx+360 pp.+xxi. (with appendix)},
   MRCLASS = {46.00 (28.00)},
  MRNUMBER = {250011},
MRREVIEWER = {H.\ E.\ Lacey},
}

@article{Chou,
    AUTHOR = {Chou, Ching},
     TITLE = {Minimal sets and ergodic measures for {$\beta N\backslash N$}},
   JOURNAL = {Illinois J. Math.},
  FJOURNAL = {Illinois Journal of Mathematics},
    VOLUME = {13},
      YEAR = {1969},
     PAGES = {777--788},
      ISSN = {0019-2082},
   MRCLASS = {28.70},
  MRNUMBER = {249569},
MRREVIEWER = {R.\ A.\ Raimi},
       URL = {http://projecteuclid.org/euclid.ijm/1256053439},
}

@article{Fairchild,
    AUTHOR = {Fairchild, Lonnie},
     TITLE = {Extreme invariant means without minimal support},
   JOURNAL = {Trans. Amer. Math. Soc.},
  FJOURNAL = {Transactions of the American Mathematical Society},
    VOLUME = {172},
      YEAR = {1972},
     PAGES = {83--93},
      ISSN = {0002-9947,1088-6850},
   MRCLASS = {43A07 (28A70)},
  MRNUMBER = {308685},
MRREVIEWER = {E.\ Granirer},
       DOI = {10.2307/1996334},
       URL = {https://doi.org/10.2307/1996334},
}

@article{Barto,
    AUTHOR = {Bartoszy\'nski, Tomek},
     TITLE = {On the structure of measurable filters on a countable set},
   JOURNAL = {Real Anal. Exchange},
  FJOURNAL = {Real Analysis Exchange},
    VOLUME = {17},
      YEAR = {1991/92},
    NUMBER = {2},
     PAGES = {681--701},
      ISSN = {0147-1937,1930-1219},
   MRCLASS = {28A05 (04A20 54E52)},
  MRNUMBER = {1171408},
MRREVIEWER = {Marek\ Balcerzak},
}

@article{Blass73,
    AUTHOR = {Blass, Andreas},
     TITLE = {The {R}udin-{K}eisler ordering of {$P$}-points},
   JOURNAL = {Trans. Amer. Math. Soc.},
  FJOURNAL = {Transactions of the American Mathematical Society},
    VOLUME = {179},
      YEAR = {1973},
     PAGES = {145--166},
      ISSN = {0002-9947,1088-6850},
   MRCLASS = {02H20 (04A20)},
  MRNUMBER = {354350},
MRREVIEWER = {F.\ G.\ Asenjo},
       DOI = {10.2307/1996495},
       URL = {https://doi.org/10.2307/1996495},
}

@article{Heike,
    AUTHOR = {Mildenberger, Heike},
     TITLE = {Exactly two and exactly three near-coherence classes},
   JOURNAL = {J. Math. Log.},
  FJOURNAL = {Journal of Mathematical Logic},
    VOLUME = {24},
      YEAR = {2024},
    NUMBER = {1},
     PAGES = {Paper No. 2350003, 41},
      ISSN = {0219-0613,1793-6691},
   MRCLASS = {03E05 (03E35 05C55)},
  MRNUMBER = {4693985},
       DOI = {10.1142/S0219061323500034},
       URL = {https://doi.org/10.1142/S0219061323500034},
}

@article{Mathias,
    AUTHOR = {Mathias, A. R. D.},
     TITLE = {Happy families},
   JOURNAL = {Ann. Math. Logic},
  FJOURNAL = {Annals of Mathematical Logic},
    VOLUME = {12},
      YEAR = {1977},
    NUMBER = {1},
     PAGES = {59--111},
      ISSN = {0003-4843},
   MRCLASS = {04A20 (02K05 02K30 04A15)},
  MRNUMBER = {491197},
MRREVIEWER = {James\ Baumgartner},
       DOI = {10.1016/0003-4843(77)90006-7},
       URL = {https://doi.org/10.1016/0003-4843(77)90006-7},
}

@article{Plachky,
    AUTHOR = {Plachky, Detlef},
     TITLE = {Extremal and monogenic additive set functions},
   JOURNAL = {Proc. Amer. Math. Soc.},
  FJOURNAL = {Proceedings of the American Mathematical Society},
    VOLUME = {54},
      YEAR = {1976},
     PAGES = {193--196},
      ISSN = {0002-9939,1088-6826},
   MRCLASS = {28A10},
  MRNUMBER = {419711},
MRREVIEWER = {A.\ G. A. G. Babiker},
       DOI = {10.2307/2040783},
       URL = {https://doi.org/10.2307/2040783},
}

@article{Blass-NCF,
    AUTHOR = {Blass, Andreas},
     TITLE = {Near coherence of filters. {I}. {C}ofinal equivalence of
              models of arithmetic},
   JOURNAL = {Notre Dame J. Formal Logic},
  FJOURNAL = {Notre Dame Journal of Formal Logic},
    VOLUME = {27},
      YEAR = {1986},
    NUMBER = {4},
     PAGES = {579--591},
      ISSN = {0029-4527,1939-0726},
   MRCLASS = {03E05 (03C20 03C62 03E35)},
  MRNUMBER = {867002},
MRREVIEWER = {Martin\ Weese},
       DOI = {10.1305/ndjfl/1093636772},
       URL = {https://doi.org/10.1305/ndjfl/1093636772},
}

@article{Miller,
    AUTHOR = {Miller, Arnold W.},
     TITLE = {There are no {$Q$}-points in {L}aver's model for the {B}orel
              conjecture},
   JOURNAL = {Proc. Amer. Math. Soc.},
  FJOURNAL = {Proceedings of the American Mathematical Society},
    VOLUME = {78},
      YEAR = {1980},
    NUMBER = {1},
     PAGES = {103--106},
      ISSN = {0002-9939,1088-6826},
   MRCLASS = {03E35 (03E05)},
  MRNUMBER = {548093},
MRREVIEWER = {A.\ Kanamori},
       DOI = {10.2307/2043048},
       URL = {https://doi.org/10.2307/2043048},
}

@article{Mirna,
    AUTHOR = {Borodulin-Nadzieja, Piotr and D{\v{z}}amonja, Mirna},
     TITLE = {On the isomorphism problem for measures on {B}oolean algebras},
   JOURNAL = {J. Math. Anal. Appl.},
  FJOURNAL = {Journal of Mathematical Analysis and Applications},
    VOLUME = {405},
      YEAR = {2013},
    NUMBER = {1},
     PAGES = {37--51},
      ISSN = {0022-247X},
   MRCLASS = {03E15 (03G05 06E25)},
  MRNUMBER = {3053484},
MRREVIEWER = {M{\'a}rton Elekes},
       DOI = {10.1016/j.jmaa.2013.03.053},
       URL = {http://dx.doi.org/10.1016/j.jmaa.2013.03.053},
}

@article{PbnDamian,
    AUTHOR = {Borodulin-Nadzieja, Piotr and Sobota, Damian},
     TITLE = {There is a {P}-measure in the random model},
   JOURNAL = {Fund. Math.},
  FJOURNAL = {Fundamenta Mathematicae},
    VOLUME = {262},
      YEAR = {2023},
    NUMBER = {3},
     PAGES = {235--257},
      ISSN = {0016-2736,1730-6329},
   MRCLASS = {03E35 (03E05 03E75 28E15)},
  MRNUMBER = {4644698},
       DOI = {10.4064/fm277-3-2023},
       URL = {https://doi.org/10.4064/fm277-3-2023},
}

@inproceedings{Solovay,
    AUTHOR = {Solovay, Robert M.},
     TITLE = {Real-valued measurable cardinals},
 BOOKTITLE = {Axiomatic set theory ({P}roc. {S}ympos. {P}ure {M}ath., {V}ol.
              {XIII}, {P}art {I}, {U}niv. {C}alifornia, {L}os {A}ngeles,
              {C}alif., 1967)},
     PAGES = {397--428},
      YEAR = {1971},
   MRCLASS = {02K35},
  MRNUMBER = {0290961},
MRREVIEWER = {Thomas J. Jech},
}

@book{sikorski,
	AUTHOR = {Sikorski, Roman},
	TITLE = {Boolean algebras},
	YEAR = {1969},
    DOI = {10.1007/978-3-642-85820-8},
	PLACE = {New York},
	PUBLISHER = {Springer},
}

@article{Cancino,
    AUTHOR = {Borodulin-Nadzieja, Piotr and Cancino-Manr\'iquez, Jonathan
              and Morawski, Adam},
     TITLE = {P-measures in models without {P}-points},
   JOURNAL = {Ann. Pure Appl. Logic},
  FJOURNAL = {Annals of Pure and Applied Logic},
    VOLUME = {176},
      YEAR = {2025},
    NUMBER = {7},
     PAGES = {Paper No. 103579},
      ISSN = {0168-0072,1873-2461},
   MRCLASS = {03E05 (03E35 03E75 28E15)},
  MRNUMBER = {4880933},
       DOI = {10.1016/j.apal.2025.103579},
       URL = {https://doi.org/10.1016/j.apal.2025.103579},
}

@incollection{Laflamme,
    AUTHOR = {Laflamme, Claude},
     TITLE = {Filter games and combinatorial properties of strategies},
 BOOKTITLE = {Set theory ({B}oise, {ID}, 1992--1994)},
    SERIES = {Contemp. Math.},
    VOLUME = {192},
     PAGES = {51--67},
 PUBLISHER = {Amer. Math. Soc., Providence, RI},
      YEAR = {1996},
      ISBN = {0-8218-0306-9},
   MRCLASS = {03E05 (03E15 03E35 04A20 90D44)},
  MRNUMBER = {1367134},
MRREVIEWER = {W.\ W.\ Comfort},
       DOI = {10.1090/conm/192/02348},
       URL = {https://doi.org/10.1090/conm/192/02348},
}

@incollection{Blass,
author = {Blass, Andreas}, 
title = {Combinatorial cardinal characteristics of the continuum}, 
booktitle = {Handbook of Set Theory}, 
pages = {395--491}, 
editor = {Foreman, M. and Kanamori, A.},
publisher = {Springer}, 
year = {2010},
}

@article{Kunen,
    AUTHOR = {Kunen, Kenneth},
     TITLE = {Some points in {$\beta N$}},
   JOURNAL = {Math. Proc. Cambridge Philos. Soc.},
  FJOURNAL = {Mathematical Proceedings of the Cambridge Philosophical
              Society},
    VOLUME = {80},
      YEAR = {1976},
    NUMBER = {3},
     PAGES = {385--398},
      ISSN = {0305-0041},
   MRCLASS = {04A15 (02K05 54A25 54D35 54G05)},
  MRNUMBER = {427070},
MRREVIEWER = {W. W. Comfort},
       DOI = {10.1017/S0305004100053032},
       URL = {https://doi.org/10.1017/S0305004100053032},
}

@article{Rudin,
	AUTHOR = {Rudin, Walter},
	TITLE = {Continuous Functions on Compact Spaces Without Perfect Subsets}, 
	JOURNAL = {Proc. Amer. Math. Soc.},
	VOLUME = {8},
	YEAR = {1957},
	PAGES = {39--42},
}

@article{Kunisada,
    AUTHOR = {Kunisada, Ryoichi},
     TITLE = {On the additive property of finitely additive measures},
   JOURNAL = {J. Theoret. Probab.},
  FJOURNAL = {Journal of Theoretical Probability},
    VOLUME = {35},
      YEAR = {2022},
    NUMBER = {3},
     PAGES = {1782--1794},
      ISSN = {0894-9840,1572-9230},
   MRCLASS = {28E10 (28C15)},
  MRNUMBER = {4488558},
       DOI = {10.1007/s10959-021-01115-3},
       URL = {https://doi.org/10.1007/s10959-021-01115-3},
}

@article{mekler,
    AUTHOR = {Mekler, Alan H.},
     TITLE = {Finitely additive measures on {${\rm {\bf N}}$} and the
              additive property},
   JOURNAL = {Proc. Amer. Math. Soc.},
  FJOURNAL = {Proceedings of the American Mathematical Society},
    VOLUME = {92},
      YEAR = {1984},
    NUMBER = {3},
     PAGES = {439--444},
      ISSN = {0002-9939},
   MRCLASS = {28A10 (03E35)},
  MRNUMBER = {759670},
MRREVIEWER = {G. A. Sokhadze},
       DOI = {10.2307/2044852},
       URL = {https://doi.org/10.2307/2044852},
}

@article{grebik,
    AUTHOR = {Greb\'{\i}k, Jan},
     TITLE = {Ultrafilter extensions of asymptotic density},
   JOURNAL = {Comment. Math. Univ. Carolin.},
  FJOURNAL = {Commentationes Mathematicae Universitatis Carolinae},
    VOLUME = {60},
      YEAR = {2019},
    NUMBER = {1},
     PAGES = {25--37},
      ISSN = {0010-2628},
   MRCLASS = {03E05 (03E35 28A12)},
  MRNUMBER = {3946662},
MRREVIEWER = {Diego Alejandro Mej\'{\i}a},
       DOI = {10.14712/1213-7243.2015.279},
       URL = {https://doi.org/10.14712/1213-7243.2015.279},
}

@incollection{Scheepers,
    AUTHOR = {Scheepers, Marion},
     TITLE = {Gaps in {$\omega^\omega$}},
 BOOKTITLE = {Set theory of the reals ({R}amat {G}an, 1991)},
    SERIES = {Israel Math. Conf. Proc.},
    VOLUME = {6},
     PAGES = {439--561},
 PUBLISHER = {Bar-Ilan Univ.},
   ADDRESS = {Ramat Gan},
      YEAR = {1993},
   MRCLASS = {03E05 (03E35 03E50 03E65 04-02 06A07)},
  MRNUMBER = {1234288 (95a:03061)},
MRREVIEWER = {Pierre Matet},
}

@article{blassple,
    AUTHOR = {Blass, Andreas and Frankiewicz, Ryszard and Plebanek, Grzegorz and
              Ryll-Nardzewski, Czes{\l}aw},
     TITLE = {A note on extensions of asymptotic density},
   JOURNAL = {Proc. Amer. Math. Soc.},
  FJOURNAL = {Proceedings of the American Mathematical Society},
    VOLUME = {129},
      YEAR = {2001},
    NUMBER = {11},
     PAGES = {3313--3320},
      ISSN = {0002-9939},
   MRCLASS = {28A12 (03E05 03E35 11B05)},
  MRNUMBER = {1845008},
MRREVIEWER = {K. P. S. Bhaskara Rao},
       DOI = {10.1090/S0002-9939-01-05941-X},
       URL = {https://doi.org/10.1090/S0002-9939-01-05941-X},
}

@article{FremlinTalagrand,
    AUTHOR = {Fremlin, David H. and Talagrand, Michel},
     TITLE = {A decomposition theorem for additive set-functions, with
              applications to {P}ettis integrals and ergodic means},
   JOURNAL = {Math. Z.},
  FJOURNAL = {Mathematische Zeitschrift},
    VOLUME = {168},
      YEAR = {1979},
    NUMBER = {2},
     PAGES = {117--142},
      ISSN = {0025-5874},
   MRCLASS = {28A10 (28B05 28C10 46G10)},
  MRNUMBER = {544700},
MRREVIEWER = {K. Jacobs},
       DOI = {10.1007/BF01214191},
       URL = {https://doi.org/10.1007/BF01214191},
}

@article{Talagrand,
    AUTHOR = {Talagrand, Michel},
     TITLE = {S\'{e}parabilit\'{e} vague dans l'espace des mesures sur un compact},
   JOURNAL = {Israel J. Math.},
  FJOURNAL = {Israel Journal of Mathematics},
    VOLUME = {37},
      YEAR = {1980},
    NUMBER = {1-2},
     PAGES = {171--180},
      ISSN = {0021-2172},
   MRCLASS = {46E27},
  MRNUMBER = {599312},
MRREVIEWER = {Richard Haydon},
       DOI = {10.1007/BF02762878},
       URL = {https://doi.org/10.1007/BF02762878},
}

@book{Bartoszynski,
    AUTHOR = {Bartoszy\'{n}ski, Tomek and Judah, Haim},
     TITLE = {Set theory},
      NOTE = {On the structure of the real line},
 PUBLISHER = {A K Peters, Ltd., Wellesley, MA},
      YEAR = {1995},
     PAGES = {xii+546},
      ISBN = {1-56881-044-X},
   MRCLASS = {03-02 (03Exx)},
  MRNUMBER = {1350295},
MRREVIEWER = {Eva Coplakova},
}

@article{Bell1981,
	author = {Bell, Murray},
	journal = {Fundamenta Mathematicae},
	language = {eng},
	number = {2},
	pages = {149-157},
	title = {On the Combinatorial Principle P(c)},
	url = {http://eudml.org/doc/211293},
	volume = {114},
	year = {1981},
}

@article{Ketonen1976,
	author = {Ketonen, Jussi},
	journal = {Fundamenta Mathematicae},
	language = {eng},
	number = {2},
	pages = {91-94},
	title = {On the existence of P-points in the Stone-\v{C}ech compactification of integers},
	url = {http://eudml.org/doc/214944},
	volume = {92},
	year = {1976},
}

@Article{Plebanek2024,
author={Plebanek, Grzegorz},
title={A survey on topological properties of P(K) spaces},
journal={Japanese Journal of Mathematics},
year={2024},
month={Oct},
day={01},
volume={19},
number={2},
pages={143-180},
issn={1861-3624},
doi={10.1007/s11537-024-2410-y},
url={https://doi.org/10.1007/s11537-024-2410-y}
}

@article {canjargeneric,
    AUTHOR = {Canjar, R. Michael},
     TITLE = {On the generic existence of special ultrafilters},
   JOURNAL = {Proc. Amer. Math. Soc.},
  FJOURNAL = {Proceedings of the American Mathematical Society},
    VOLUME = {110},
      YEAR = {1990},
    NUMBER = {1},
     PAGES = {233--241},
      ISSN = {0002-9939,1088-6826},
   MRCLASS = {03E05 (03E65 04A20 54A25)},
  MRNUMBER = {993747},
MRREVIEWER = {Marion\ Scheepers},
       DOI = {10.2307/2048264},
       URL = {https://doi.org/10.2307/2048264},
}

\end{document}